\documentclass[a4paper]{amsart}
\usepackage{amssymb}

\usepackage{mathrsfs}
\newtheorem{theorem}{Theorem}
\newtheorem{lem}[theorem]{Lemma}
\newtheorem{cor}[theorem]{Corollary}
\newtheorem{prop}[theorem]{Proposition}
\theoremstyle{definition}
\newtheorem{defi}[theorem]{Definition}
\newtheorem{example}[theorem]{Example}

\theoremstyle{remark}
\newtheorem{rem}[theorem]{Remark}

\newcommand{\BC}{\mathbb{C}}            %% field of complex numbers
\newcommand{\BZ}{\mathbb{Z}}             %% ring of natural numbers  

\newcommand{\Glie}{\mathfrak{g}}        %% Lie algebra
\newcommand{\Hlie}{\mathfrak{h}}          %% Cartan subalgebra
\newcommand{\Yang}{Y_{\hbar}(\mathfrak{g})}    %% Yangian
\newcommand{\BGG}{\mathcal{O}}      %% category BGG
\newcommand{\HJ}{\mathcal{O}_{\mathrm{HJ}}}

\newcommand{\BQ}{\mathbf{Q}}                %% weight lattice
\newcommand{\wt}{\mathrm{wt}}         %% the set of weights

\newcommand{\CL}{\mathscr{L}}       %% affine weights
\newcommand{\CR}{\mathscr{R}}     %% classification set
\newcommand{\lCQ}{\mathcal{Q}}     % affine root lattice
\newcommand{\lCP}{\mathcal{P}}     % affine weight lattice
\newcommand{\lwt}{\mathrm{wt}_{\ell}}   %% the set of l-weights
\newcommand{\Be}{\mathbf{e}}
\newcommand{\Bn}{\mathbf{n}}  
\newcommand{\Bm}{\mathbf{m}} 
\newcommand{\Bf}{\mathbf{f}}  
\newcommand{\Bw}{\mathbf{w}}    %% KR
\newcommand{\Bd}{\mathbf{d}}    %% Demazure
\newcommand{\Bs}{\mathbf{s}}
\newcommand{\Bt}{\mathbf{t}}   

\newcommand{\CEl}{\mathcal{E}_{\ell}}        %% q-character ring
\newcommand{\qc}{\chi_{\mathrm{q}}}            %% q-character
\newcommand{\nqc}{\widetilde{\qc}}            %% normalized q-character
\newcommand{\pr}{\mathrm{pr}}           %% projection to l-weight spaces

\newcommand{\CM}{\mathscr{M}}     %% prime Demazure
\newcommand{\SF}{\mathscr{F}}         %% inductive map
\newcommand{\SG}{\mathscr{G}}     %% inductive map difference
\newcommand{\SL}{\mathscr{L}}     %% asymptotic formula

\allowdisplaybreaks                %% equations allowed to be breaked

\begin{document}
\title{Yangians and Baxter's relations}
\author{Huafeng Zhang}
\thanks{Univ. Lille, CNRS, UMR 8524 - Laboratoire Paul Painlev\'e, F-59000 Lille, France \\
Email: huafeng.zhang@univ-lille.fr } 

\begin{abstract}
We study a category $\mathcal{O}$ of representations of the Yangian associated to an arbitrary finite-dimensional complex simple Lie algebra. We obtain asymptotic modules as analytic continuation of a family of finite-dimensional modules, the Kirillov--Reshetikhin modules. In the Grothendieck ring we establish the three-term Baxter's TQ relations for the asymptotic modules. We indicate that Hernandez--Jimbo's limit construction can also be applied, resulting in modules over anti-dominantly shifted Yangians.


\end{abstract}
\maketitle
%\tableofcontents

\section*{Introduction}

Fix $\Glie$ to be a finite-dimensional complex simple Lie algebra, and $\hbar$ a non-zero complex number. The universal enveloping algebra of the current algebra $\Glie \otimes_{\BC} \BC[t]$, as a co-commutative Hopf algebra, can be deformed to the {\it Yangian} $\Yang$. The latter is a Hopf algebra neither commutative nor co-commutative, and it contains the universal enveloping algebra of $\Glie$ as a sub-Hopf-algebra. 

In this paper we are interested in a category $\BGG$ of representations of $\Yang$. Its objects are $\Yang$-modules which, viewed as $\Glie$-modules, belong to the category of Bernstein--Gelfand--Gelfand (BGG for short)  without integrability assumption \cite{Kac}. Category $\BGG$ contains all finite-dimensional modules, it is abelian and monoidal, and its Grothendieck ring is commutative thanks to a highest weight classification of irreducible objects, a common phenomenon in Lie Theory. 

As a prototype, consider the Lie algebra $\mathfrak{sl}_2$ of traceless two-by-two matrices. It admits the vector representation on $\mathbb{C}^2$ by matrix multiplication, and the infinite-dimensional Verma module $M_x$ of highest weight $x \in \BC$. In the Grothendieck ring of the category BGG of $\mathfrak{sl}_2$-modules one 
has 
\begin{equation} \label{equ: sl2}
[\BC^2]  [M_x] = [M_{x+1}] + [M_{x-1}] \quad \forall x \in \BC.
\end{equation}
Indeed, the Verma module $M_{x+1}$ is embedded in the tensor product $\BC^2 \otimes M_x$ via highest weight vectors, the quotient of which is isomorphic to $M_{x-1}$.

Our main result of this paper, Theorem \ref{thm: main}, describes similar {\bf three-term} relations for the Yangian $\Yang$. Namely, we find Yangian analogs of: 
\begin{itemize}
\item the Verma modules  $\rightsquigarrow$ the {\it asymptotic modules} $\SL(\frac{\Psi_{i,y}}{\Psi_{i,x}})$ in Definition \ref{defi: asymptotic module};
\item the two-dimensional module  $\rightsquigarrow$ the modules $\CM_{k,x}^{(i)}$ in Equation \eqref{def: generic prime demazure}.
\end{itemize}
Here $k, x, y$ are complex numbers, $i$ is a Dynkin node of the underlying Lie algebra $\Glie$, and the $M_{x\pm 1}$ in the Yangian situation are tensor products of asymptotic modules. As a byproduct,
we obtain three-term identities in the Grothendieck ring of the category BGG of $\Glie$-modules, which in the case $\Glie = \mathfrak{sl}_2$ reduce to Equation \eqref{equ: sl2}.  

Besides,  we derive three-term relations for the quantum affine algebra $U_q(\widehat{\Glie})$ where $q$ is a generic complex number. The required asymptotic modules were constructed earlier in \cite{Z2}, and they belong to a category $\hat{\BGG}$ of $U_q(\widehat{\Glie})$-modules introduced in \cite{He} and further studied in \cite{MY}. 

\

One of our motivations behind these relations lies in the spectral problem of exactly solvable models. For such a model, it is essential to determine the common spectra of a family $T(z)$, with a complex parameter $z$, of commuting endomorphisms of a vector space. An important idea of R.~Baxter \cite{Baxter72} is to construct auxiliary endomorphisms $Q(z)$ that mutually commute with the $T(z)$ and obey a functional relation, the {\it Baxter's TQ relation}:
\begin{equation}  \label{Baxter TQ}
T(z) Q(z) = f(z) Q(z+\hbar) + g(z) Q(z-\hbar)
\end{equation}  
for certain scalar functions $f(z)$ and $g(z)$. Consider the Heisenberg spin chain as an example. The $T(z)$ and $Q(z)$ are polynomials in $z$, so that the eigenvalues of $Q(z)$ have the form $\prod_{i=1}^n(z-z_i)$. Under genericity assumption on the roots, by inserting $z = z_i$ in the TQ relation we obtain the {\it Bethe Ansatz Equations}
$$ \prod_{j: j\neq i} \frac{z_i-z_j+\hbar}{z_i-z_j -\hbar} = \frac{g(z_i)}{f(z_i)},\quad 1 \leq i \leq n. $$
Eigenvalues of $T(z)$ are expressed in terms of the roots $z_i$ based on the TQ relation.

After the pioneer works of Bazhanov--Lukyanov--Zamolodchikov \cite{BazhanovLukyanovZamolodchikov1997, BazhanovLukyanovZamolodchikov1999}, the recent work of Frenkel--Hernandez \cite{FH} gave a representation-theoretic interpretation of Baxter's TQ relation. One starts from a quantum affine algebra $U_q(\widehat{\Glie})$ with a fixed representation $V$, referred to as quantum space. Hernandez--Jimbo \cite{HJ} introduced a category $\HJ$ of representations of the upper Borel subalgebra whose Grothendieck ring $K_0(\HJ)$ is commutative. The standard transfer matrix construction based on the universal R-matrix endows $V$ with an action of $K_0(\HJ)$. 

This already recovers the well-known six-vertex model. The Lie algebra $\Glie$ is $\mathfrak{sl}_2$. For $a \in \BC^{\times}$, one has a two-dimensional irreducible representation $\BC_a^2$ of $U_q(\widehat{\mathfrak{sl}_2})$ (coming from Jimbo's evaluation map $U_q(\widehat{\mathfrak{sl}_2}) \longrightarrow U_q(\mathfrak{sl}_2)$) and the so-called positive prefundamental module $L_a^+$ over the upper Borel subalgebra. The transfer matrix of $\BC_a^2$ gives $T(za)$, while that of $L_a^+$ gives $Q(za)$.  Baxter's TQ relation is a consequence of the following identity in $K_0(\HJ)$:
$$ [\BC_a^2] [L_a^+] = [\BC_q][L_{aq^{-2}}^+]  + [\BC_{q^{-1}}][L_{aq^2}^+]\quad \forall a \in \BC^{\times}. $$
Here $\BC_q$ and $\BC_{q^{-1}}$ are one-dimensional weight modules in category $\HJ$. Their transfer matrices are scalar products by $f(z)$ and $g(z)$ in the TQ relation.

For higher rank Lie algebras, various three-term identities in $K_0(\HJ)$ have been established \cite{Jimbo,FH2,HL}, the right-hand side of which is a sum of two monomials in the prefundamental modules. This leads to a proof of Bethe Ansatz Equations for the quantum integrable system attached to $U_q(\widehat{\Glie})$.

Let us return to Heisenberg spin chain. The quantum group in question is the Yangian $Y_{\hbar}(\mathfrak{gl}_2)$. One does not have Borel subalgebras to define category $\HJ$. Indeed, from the viewpoint of Yangian double \cite{KT}, the Yangian itself is a Borel subalgebra. Still there are  prefundamental modules $L_a^{\pm}$ over the degenerate Yangian \cite{B,B2}. \footnote{The Yangian $Y_{\hbar}(\mathfrak{gl}_2)$ has the R-matrix realization. Its generators are encoded in a square matrix $T(z)$ whose entries are power series of $z^{-1}$ with initial condition that $T(\infty)$ is the identity matrix. Dropping the initial condition, one obtains the degenerate Yangian. \label{footnote} } Baxter's TQ relation \eqref{Baxter TQ} comes from a similar identity
\begin{equation}  \label{TQ}
 [\BC_a^2][L_a^+] = [\theta_a][L_{a+\hbar}^+]  + [\theta_a'][L_{a-\hbar}^+] \quad \forall a \in \BC,
\end{equation}
where $\BC_a^2$ is a two-dimensional irreducible module over $Y_{\hbar}(\mathfrak{gl}_2)$ indexed by $a \in \BC$, and $\theta_a$ and $\theta_a'$ are one-dimensional modules over the degenerate Yangian. 

\

Equations \eqref{equ: sl2} and \eqref{TQ} bear strong resemblance. In this paper we argue that the prefundamental modules over degenerate Yangian can be replaced by the asymptotic modules over the ordinary Yangian. One main reason is that the right-hand side of the three-term identity in Theorem \ref{thm: main} is a sum of two monomials in asymptotic modules. This should eventually lead to Bethe Ansatz Equations for quantum integrable system attached to $\Yang$, upon identification of the Q-operators with transfer matrices of asymptotic modules. See \cite[Appendix A]{FZ} for the $\mathfrak{sl}_2$ case.

We give a few comments on the asymptotic modules of Yangians. Let $I$ be the set of Dynkin nodes of $\Glie$. It is a classical result of V.~Drinfeld \cite{Dr} that finite-dimensional irreducible $\Yang$-modules are parametrized by $I$-tuples of rational functions of $u$ satisfying a polynomiality property. For $(i, k, x) \in I \times \BZ_{>0} \times \BC$ one has the Kirillov--Reshetikhin module \cite{KR} corresponding to the $I$-tuple
$$ (1,\cdots, 1, \frac{u+kd_i\hbar+x\hbar}{u+x\hbar}, 1, \cdots, 1) $$
with the non-trivial function at the $i$-th position. The asymptotic module $\SL(\frac{\Psi_{i,y}}{\Psi_{i,x}})$ for $y \in \BC$ is an analytic continuation of these modules: as the integer $k$ goes to infinity, one replaces $kd_i$ with $y-x$. This is inspired by earlier limit construction for quantum affine algebras \cite{HJ} where $q^{-k}$ specializes to $0$ as $k$ goes to infinity.

Our approach has the advantage that all representations are defined over the full quantum group. It was initiated in a joint work with G.~Felder \cite{FZ} and further developed in \cite{Z4} for elliptic quantum groups of type A. The essential part of the proof of Baxter's TQ relation in \cite{Z4} is a short exact sequence of tensor products of Kirillov--Reshetikhin modules and a Demazure-like module. Such exact sequence exists for Yangians \cite{FoH}. When the underlying Lie algebra $\Glie$ is simply laced, the proof of \cite{Z4} is readily adapted to Yangians. While for $\Glie$ non simply laced, a tensor product factorization of this Demazure-like module is necessary; see Remark \ref{rem: factorization}. \footnote{Readers interested in the simply-laced case can ignore Proposition \ref{prop: generalized simple roots}, Corollary \ref{cor: layer structure of l-weights}, and the key technical result Corollary \ref{cor: l-weight prime Demazure}, as the tensor product factorization of Remark \ref{rem: factorization} is trivial. }

Additionally, we apply Hernandez--Jimbo's limit procedure \cite{HJ} to obtain prefundamental modules over the so-called shifted Yangians introduced in  \cite{BFN} (roughly speaking one replaces the R-matrix realization in Footnote \ref{footnote} by the Drinfeld current realization). Our negative prefundamental modules have the same q-characters as those in \cite{HJ},  yet strangely the positive ones are always one-dimensional.

We expect the asymptotic modules as well as their three-term relations to exist for elliptic quantum groups outside of type A \cite{F,GTL2,YZ}, one of the main technical difficulties being the lack of coproduct. As an intermediate step, it is interesting to study quantum toroidal algebras and affine Yangians whose Drinfeld--Jimbo type coproduct has been partly established \cite{GNW, JZ}. 

This paper is structured as follows. In Section \ref{sec: pre} we define the category $\BGG$ of $\Yang$-modules and study its basic properties. Section \ref{sec: KR} deduces some combinatorial properties of Kirillov--Reshetikhin modules from those in the case of quantum affine algebras via the functor of Gautam--Toledano Laredo. In Section \ref{sec: asym} we construct asymptotic modules and prove the three-term identities. In Section \ref{sec: prefund} we define prefundamental modules over shifted Yangians. In the appendix we provide the three-term identities for quantum affine algebras. 

\

{\it Acknowledgments.} The author thanks Giovanni Felder, David Hernandez, Nicolai Reshetikhin and Marc Rosso for helpful discussions. 

\section{Preliminaries}  \label{sec: pre}
We collect basic facts on Yangians and their representations. All vector spaces are over $\BC$. Recall that $\Glie$ is fixed to be a finite-dimensional simple Lie algebra.

Let $\Hlie$ be the Cartan subalgebra of $\Glie$, and $I := \{1,2,\cdots,r\}$ the set of Dynkin nodes. The dual space $\Hlie^*$ admits a basis of {\it simple roots} $(\alpha_i)_{i \in I}$ and a non-degenerate symmetric bilinear form $(,): \Hlie^* \times \Hlie^* \longrightarrow \BC$. For $i,j \in I$ set 
$$c_{ij} := \frac{2(\alpha_i,\alpha_j)}{(\alpha_i,\alpha_i)} \in \BZ,\quad d_{ij} := \frac{(\alpha_i,\alpha_j)}{2},\quad d_i := d_{ii}. $$
After normalization $d_i \in \{1,2,3\}$ and their greatest common divisor is 1. (If $\Glie$ is simply laced, then $d_i = 1$ for all $i$.) Let $\BQ := \oplus_{i\in I} \BZ \alpha_i \subset \Hlie^*$ be the {\it root lattice} and set $\BQ_- := \oplus_{i\in I} \BZ_{\leq 0} \alpha_i$.  Define the {\it fundamental weights} $\varpi_i \in \Hlie^*$ for $i \in I$ by the equations $(\alpha_i,\varpi_j) = d_i \delta_{ij}$. Define the {\it height function} $h: \BQ \longrightarrow \BZ $ to be the abelian group morphism sending each $\alpha_i$ to $1$.

The Yangian $\Yang$ is an algebra generated by elements $\{x_{i,n}^{\pm}, \xi_{i,n}\}_{i \in I, n \in \BZ_{\geq 0}}$ subject to the following relations for $i,j \in I$ and $m,n \in \BZ_{\geq 0}$:
\begin{gather}
[\xi_{i,m}, \xi_{j,n}] = 0,\quad [\xi_{i,0}, x_{j,n}^{\pm}] = \pm 2d_{ij} x_{j,n}^{\pm}, \quad [x_{i,m}^+,x_{j,n}^-] = \delta_{ij} \xi_{i,m+n}, \label{rel: Cartan}  \\
[\xi_{i,m+1}, x_{j,n}^{\pm}] - [\xi_{i,m}, x_{j,n+1}^{\pm}] = \pm d_{ij}\hbar (\xi_{i,m}x_{j,n}^{\pm} + x_{j,n}^{\pm} \xi_{i,m}), \label{rel: Cartan-Drinfeld}  \\
[x_{i,m+1}^{\pm}, x_{j,n}^{\pm}] - [x_{i,m}^{\pm}, x_{j,n+1}^{\pm}] = \pm d_{ij}\hbar (x_{i,m}^{\pm}x_{j,n}^{\pm} + x_{j,n}^{\pm} x_{i,m}^{\pm}), \label{rel: Drinfeld} \\
\mathrm{ad}_{x_{i,0}^{\pm}}^{1-c_{ij}} (x_{j,n}^{\pm}) = 0 \quad \mathrm{if}\ i \neq j. \label{rel: Serre}
\end{gather}
Here $\mathrm{ad}_x(y) := xy - yx$. Define the generating series in $u^{-1}$ for $i \in I$:
 \begin{equation}  \label{def: generators currents}
x_i^{\pm}(u) := \hbar \sum_{n=0}^{\infty} x_{i,n}^{\pm} u^{-n-1},\quad \xi_i(u) := 1 + \hbar \sum_{n=0}^{\infty} \xi_{i,n} u^{-n-1}.
\end{equation}
 
 Algebra $\Yang$ is $\BQ$-graded ({\it weight grading}): the elements $\xi_{i,n}$ are of weight $0$, while the $x_{i,n}^{\pm}$ are of weight $\pm \alpha_i$ respectively. For $\beta \in \BQ$ let $\Yang_{\beta}$ denote the subspace of elements in $\Yang$ of weight $\beta$. By Relation \eqref{rel: Cartan} an element $x \in \Yang$ is of weight $\beta$ if and only if $[\xi_{i,0}, x] = (\alpha_i,\beta) x$ for all $i \in I$.

Let $Y_{\hbar}^{+}(\Glie)$, $Y_{\hbar}^-(\Glie)$ be the subalgebras of $\Yang$ generated by the $x_{i,n}^+$, $x_{i,n}^-$ respectively. Let $Y_{\hbar}^0(\Glie)$ be the commutative subalgebra of $\Yang$ generated by the $\xi_{i,n}$. We have the triangular decomposition $Y^-_{\hbar}(\Glie) Y^0_{\hbar}(\Glie) Y^+_{\hbar}(\Glie) = \Yang$.

The subalgebra generated by the $x_{i,0}^{\pm},\ \xi_{i,0}$ for $i \in I$ is isomorphic to  the enveloping algebra $U(\Glie)$ of $\Glie$, Relations \eqref{rel: Cartan} and \eqref{rel: Serre} for $m=n=0$ identified with the Serre presentation. $\Yang$ admits a co-multiplication $\Delta: \Yang \longrightarrow \Yang^{\otimes 2}$ extending that on the enveloping algebra: $\Delta(x) = 1 \otimes x + x \otimes 1$ for $x \in \{x_{i,0}^{\pm}, \xi_{i,0}\}$. 

\begin{lem}\cite{Kn}  \label{lem: coproduct estimation}
For $i \in I$ we have (summations run over $\beta \in \BQ$):
\begin{align*}
\Delta(x_{i}^+(u)) &- x_{i}^+(u) \otimes 1 - \xi_i(u) \otimes x_{i}^+(u)  \in \sum_{h(\beta) \geq 2} \Yang_{\alpha_i-\beta} \otimes \Yang_{\beta}[[u^{-1}]], \\
\Delta(x_{i}^-(u)) &- 1 \otimes x_{i}^-(u) - x_{i}^-(u) \otimes \xi_i(u)  \in \sum_{h(\beta) \geq 2} \Yang_{-\beta} \otimes \Yang_{\beta-\alpha_i}[[u^{-1}]], \\
\Delta(\xi_{i}(u)) &- \xi_{i}(u) \otimes \xi_{i}(u) \in \sum_{h(\beta) \geq 1} \Yang_{-\beta} \otimes \Yang_{\beta}[[u^{-1}]].
\end{align*}
\end{lem}

\begin{rem}  \label{rem: q mult}
Let $V$ and $W$ be $\Yang$-modules.  Suppose  $\omega \in W$ is annihilated by all the  $x_i^+(u)$. By triangular decomposition, $\Yang \omega = Y_{\hbar}^-(\Glie) Y_{\hbar}^0(\Glie) \omega$, so $\Yang_{\beta} \omega = 0$ whenever $h(\beta) > 0$. In the tensor product module $V \otimes W$, the subspace $V \otimes \omega$ is annihilated by the extra summations of Lemma \ref{lem: coproduct estimation}. 
\end{rem}

Let $V$ be a $\Yang$-module. For $\beta \in \Hlie^*$, the following subspace, if non-zero, is called a weight space (and $\beta$ a weight of $V$):
$$ V_{\beta} := \{v \in V \ |\ \xi_{i,0} v = (\alpha_i,\beta) v \quad \mathrm{for}\ i \in I \}.   $$
We have $\Yang_{\alpha} V_{\beta} \subseteq V_{\alpha+\beta}$ for $\alpha \in \BQ$. Let $\wt(V) \subseteq \Hlie^*$ be the set of weights of $V$. 

\begin{defi} \label{def: O}
$\BGG$ is a full subcategory of the category of $\Yang$-modules. An object of $\BGG$ is a $\Yang$-module $V$ subject to the following conditions: 
\begin{itemize}
\item[(O1)] it is a direct sum of finite-dimensional weight spaces;
\item[(O2)] there exist $\lambda_1,\lambda_2,\cdots,\lambda_n \in \Hlie^*$ such that $\wt(V) \subseteq \cup_{j=1}^n (\lambda_j + \BQ_-)$.
\end{itemize}
\end{defi}
Every finite-dimensional $\Yang$-module is in category $\BGG$, its weight grading guaranteed by the $\Glie$-module structure. Category $\BGG$ is abelian and equipped with a monoidal structure by the co-multiplication of $\Yang$.  

We need a refinement of weights based on spectral decomposition of the action of the commutative subalgebra $Y_{\hbar}^0(\Glie)$. Introduce the group $\CL := (1 + u^{-1}\BC[[u^{-1}]])^I$ of $I$-tuple of Laurent series in $u^{-1}$ with  component-wise multiplication. Its elements are written typically as $\Bd, \Be, \Bf$ etc. For such an element $\Be$ and for $(i, n) \in I\times \BZ_{\geq 0}$, let $\Be_i(u) \in \BC((u^{-1}))$ denote the $i$-th component of $\Be$, and $\Be_{i,n} \in \BC$ the coefficient of $u^{-n-1}$ in $\hbar^{-1} \Be_i(u)$. So $\Be_i(u) = 1 + \hbar \sum_{n\geq 0} \Be_{i,n} u^{-n-1}$. 

The following map, called {\it weight projection}, is an abelian group morphism:
$$\varpi: \CL \longrightarrow \Hlie^*,\quad \Be \mapsto \sum_{i\in I} \frac{\Be_{i,0}}{d_i} \varpi_i.  $$

Let $V$ be a $\Yang$-module. For $\Be \in \CL$, the following subspace, if non-zero, is called an $\ell$-weight space (and $\Be$ an $\ell$-weight of $V$)
$$ V_{\Be} :=   \{v \in V \ |\ \forall\ (i,n) \in I \times \BZ_{\geq 0}\ \exists\ m \in \BZ_{>0}\ \mathrm{such\ that}\ (\xi_{i,n} - \Be_{i,n})^m v = 0  \}. $$
Let $\lwt(V) \subseteq \CL$ denote the set of $\ell$-weights of $V$. 

A non-zero vector $v \in V$ is called a {\it highest $\ell$-weight vector} if it is a common eigenvector of the $\xi_{i,n}$ and is annihilated by the $x_{i,n}^+$. If furthermore $V = \Yang v$, then $V$ is called a {\it highest $\ell$-weight module}. In this situation, by Remark \ref{rem: q mult}, $v$ belongs to a one-dimensional $\ell$-weight space $V_{\Bd}$ which is also the weight space $V_{\varpi(\Bd)}$ (call $\Bd$ the highest $\ell$-weight of $V$), and we have $\wt(V) = \varpi(\Bd) + \BQ_-$. The tensor product of two highest $\ell$-weight vectors  is a highest $\ell$-weight vector.

If $V$ is in category $\BGG$, then each weight space $V_{\beta}$ is a direct sum of the $V_{\Be}$ where $\Be \in \lwt(V)$ and $\varpi(\Be) = \beta$. 
Following H.~Knight \cite{Kn}, define its {\it q-character} by 
$$ \qc(V) := \sum_{\Be \in \lwt(V)}  \dim(V_{\Be}) \Be \in \CEl. $$
 The target $\CEl$ is the set of formal sums $\sum_{\Be \in \CL} n_{\Be} \Be$ of the symbols $\Be$ with integer coefficients $n_{\Be}$ subject to the following conditions \cite[\S 3.4]{HJ}
\begin{itemize}
\item[(E1)] for each $\beta \in \Hlie^*$ the set $\{\Be \in \CL\ |\ n_{\Be} \neq 0,\ \varpi(\Be) = \beta \}$ is finite;
\item[(E2)] there exist $\lambda_1,\lambda_2,\cdots,\lambda_m \in \Hlie^*$ such that $\varpi(\Be) \in \cup_{j=1}^m (\lambda_j + \BQ_-)$ if $n_{\Be} \neq 0$.
\end{itemize}
It is a ring: addition is the usual one of formal sums; multiplication is induced by that of $\CL$. (One views $\CEl$ as a completion of the group ring $\BZ[\CL]$.) 

\begin{rem}  \label{rem: spectral shift}
To $x \in \BC$ one attaches a Hopf algebra automorphism:
\begin{equation}  \label{def: spectral shift}
\tau_x: \Yang \longrightarrow \Yang,\quad x_{i}^{\pm}(u) \mapsto x_i^{\pm}(u+x\hbar),\quad \xi_i(u) \mapsto \xi_i(u+x\hbar)
\end{equation}
called {\it spectral parameter shift}.
The pullback of a module $V$ by $\tau_a$ is another module whose q-character is obtained from $\qc(V)$ by replacing each of the $\ell$-weights $(\Be_i(u))_{i \in I}$ of $V$ with $(\Be_i(u+x\hbar))_{i\in I}$. 
\end{rem}

Define $\CR'$ to be the subset of $\CL$ whose elements are highest $\ell$-weights of highest $\ell$-weight modules in category $\BGG$. For $\Bd \in \CR'$, there exists a unique (up to isomorphism) irreducible module in category $\BGG$ of highest $\ell$-weight $\Bd$, denoted by $L(\Bd)$. Based on the triangular decomposition, the $L(\Bd)$ for $\Bd \in \CR'$ form the set of (mutually non-isomorphic) irreducible modules in category $\BGG$. 

Let us view $\CL$ as a subgroup of the monoid $\BC((u^{-1}))^I$. For $(i,x) \in I \times \BC$, the following three invertible elements of $\BC((u^{-1}))^I$ will play a crucial role:
\begin{align}  
    &\Psi_{i,x} := (\underbrace{1,\cdots,1}_{i-1}, u + x \hbar, \underbrace{1,\cdots,1}_{r-i}) \quad \quad \mathrm{prefundamental\ weight},\label{def: pre-fund} \\
    & A_{i,x} := \prod_{j \in I} \frac{\Psi_{j, x + d_{ij}}}{\Psi_{j, x - d_{ij}} } \quad \quad \mathrm{generalized\ simple\ root},  \label{def: simple root} \\
    & Y_{i,x} := \frac{\Psi_{i,x+\frac{1}{2}d_i}}{\Psi_{i,x-\frac{1}{2}d_i}} \quad \quad \mathrm{fundamental\ loop\ weight}.  \label{def: fund l-weights}
\end{align}
   The ratios $\frac{\Psi_{i,x}}{\Psi_{i,y}}$ for $(i,x,y) \in I \times \BC^2$ belong to $\CL$ and they generate a subgroup, denoted by $\CR$. Equivalently, an element $\Be \in \CL$ belongs to $\CR$ if and only if each of the $\Be_i(u)$ is the Taylor expansion around $u = \infty$ of a rational function.

\begin{lem} \label{lem: denominator weight space}
Fix $\Bd \in \CR'$. Take a highest $\ell$-weight vector $\omega$ from $L(\Bd)$.
\begin{itemize}
\item[(a)] We have $\Bd \in \CR$. Fix $i \in I$ and let $u^s + c_1 u^{s-1} + \cdots + c_{s-1} u + c_s$ be the denominator of the rational function $\Bd_i(u)$.
\item[(b)] The $L(\Bd)$-valued Laurent series 
$ (u^s + c_1 u^{s-1} + \cdots + c_{s-1} u + c_s) x_i^-(u) \omega $
is a polynomial in $u$ of degree $s-1$ with leading term $\hbar x_{i,0}^- \omega u^{s-1}$. 
\end{itemize}
In the special case $\Bd_i(u) = 1$,  we have $s = 0$ and $x_i^-(u) \omega = 0$.
\end{lem}
\begin{proof}
The finite-dimensional weight space  $L(\Bd)_{\varpi(\Bd)-\alpha_i}$ being spanned by infinitely many vectors $x_{i,n}^- \omega$ for $n \in \BZ_{\geq 0}$, one finds $a_1,\cdots, a_m \in \BC$ such that 
$$ x_{i,m}^- \omega + a_{1} x_{i,m-1}^-\omega + \cdots + a_{m-1} x_{i,1}^-\omega + a_m x_{i,0}^- \omega = 0. $$
Applying the $x_{i,n}^+$ to this identity, from Relation \eqref{rel: Cartan} we obtain that the Laurent series $(u^m + a_1 u^{m-1} + \cdots + a_{m-1} u + a_m) \Bd_i(u)$ in $u^{-1}$ is a monic polynomial in $u$ of degree $m$, so $\Bd_i(u)$ is a ratio of two monic polynomials of the same degree. This proves part (a). For part (b), one checks that for $n \in \BZ_{\geq 0}$ the vector 
$$ x_{i,n+s}^- \omega + c_{1} x_{i,n+s-1}^-\omega + \cdots + c_{s-1} x_{i,n+1}^-\omega + c_s x_{i,n}^- \omega \in L(\Bd)_{\varpi(\Bd)-\alpha_i}  $$
is annihilated by all the $x_j^+(u)$ and must be zero because of irreducibility.  
\end{proof}

We shall see in Theorem \ref{thm: asymptotic module} that $\CR = \CR'$.

\begin{prop} \label{prop: generalized simple roots}
Fix $j \in I$. Let $V$ be in category $\BGG$ and  $\Be, \Bf \in \lwt(V)$.  Consider the projection $\mathrm{pr}_{\Bf}: V \longrightarrow V_{\Bf}$  with respect to the $\ell$-weight space decomposition.  If $ \pr_{\Bf}(x_{j,n}^+ V_{\Be}) \neq 0$ for certain $n \in \BZ_{\geq 0}$, then $\Bf =  A_{j,x}\Be$ for a unique $x \in \BC$.
\end{prop}
\begin{proof}
We follow the proof of \cite[Proposition 3.8]{MY}.
Choose ordered bases $(v_k)_{1\leq k \leq s}$ and $(w_l)_{1\leq l \leq t}$ of $V_{\Be}$ and $V_{\Bf}$ respectively such that for $i \in I$:
\begin{itemize}
    \item the series $\xi_i(u) v_k$ is $\Be_i(u) v_k$ plus a sum of the $v_l$ for $1\leq l < k$;
    \item the series $\xi_i(u) w_k$ is $\Bf_i(u) w_k$ plus a sum of the $w_l$ for $k < l \leq t$.
\end{itemize}
 Since $ \pr_{\Bf}(x_{j,n}^+ V_{\Be}) \neq 0$, at least one of the $\pr_{\Bf}(x_j^+(z) v_k)$ is nonzero. One can find $1\leq K \leq s,\ 1\leq L \leq t$ and $0\neq \lambda(z) \in \BC[[z^{-1}]]$ such that: 
\begin{itemize}
\item the series $\pr_{\Bf}( x_j^+(z) v_K)$ is $\lambda(z) w_L $ plus a sum of the $w_l$ for $L<l\leq t$;
\item the series $\pr_{\Bf}(x_j^+(z) v_k)$ is zero for $1\leq k < K$.
\end{itemize}
Let $Z$ be the subspace of $V_{\Bf}$ spanned by the $w_l$ for $L<l\leq t$. Then modulo the space $Z[[z^{-1}, u^{-1}]]$ we have 
$\pr_{\Bf}(x_j^+(z) \xi_i(u) v_K)  \equiv \Be_i(u)\lambda(z) w_L$.

Let us rewrite the relation \eqref{rel: Cartan-Drinfeld} in the form of generating series \cite[\S 2.3]{GTL}:
\begin{equation*}  
(u-z- d_{ij}\hbar)\xi_i(u)x_j^{+}(z) = (u-z + d_{ij}\hbar) x_j^{+}(z) \xi_i(u) - 2d_{ij}\hbar x_j^{+}(u- d_{ij}\hbar) \xi_i(u).
\end{equation*}
Apply this relation to $v_K$ and then project the resulting identity to $V_{\Bf}$. Since the projection commutes with $\xi_i(u)$, modulo the space $Z[[z^{-1},u^{-1}]]$, we have
\begin{align*}
&\mathrm{LHS} = (u-z-d_{ij}\hbar)  \xi_i(u) \pr_{\Bf}( x_j^+(z) v_K) \equiv  (u-z-d_{ij}\hbar) \Bf_i(u) \lambda(z) w_L, \\
&\mathrm{RHS} = (u-z+d_{ij}\hbar) \pr_{\Bf}(x_j^+(z)\xi_i(u) v_K ) - 2d_{ij}\hbar \pr_{\Bf}(x_j^+(u-d_{ij}\hbar) \xi_i(u) v_K) \\
&\quad\quad \equiv [ (u-z+d_{ij}\hbar) \lambda(z)- 2d_{ij}\hbar \lambda(u-d_{ij}\hbar)]\Be_i(u) w_L, \\
&(u-z-d_{ij}\hbar) \Bf_i(u) \lambda(z) =  (u-z+d_{ij}\hbar) \Be_i(u) \lambda(z) - 2d_{ij}\hbar  \Be_i(u) \lambda(u-d_{ij}\hbar).
\end{align*}
Write $\lambda(z) = \lambda_m z^{-m-1} + \lambda_{m+1}z^{-m-2} + \cdots$ with $\lambda_m \neq 0$. We take coefficients of $z^{-m-1}$ in the last equation. The third term does not contribute, and we get  
$$ \frac{\Bf_i(u)}{\Be_i(u)} = \frac{u+d_{ij}\hbar - \lambda_{m+1}\lambda_m^{-1}  }{u-d_{ij}\hbar - \lambda_{m+1} \lambda_m^{-1} } \quad \mathrm{for}\ i \in I.  $$
This implies $\Bf = A_{j,x}\Be $ with $x = - \lambda_{m+1} \lambda_m^{-1} \hbar^{-1}$.
\end{proof}

Let $\lCP_+$  and $\lCQ_-$ be the submonoids of $\CR$ generated by the $Y_{i,x}$ and the $A_{i,x}^{-1}$ for $(i,x) \in I \times \BC$ respectively. We shall also need the subgroup $\lCP$ of $\CR$ generated by the $Y_{i,x}$. 
By \cite{Dr}, $\lCP_+ \subseteq \CR'$ and the irreducible module $L(\Bd)$ is finite-dimensional if and only if $\Bd \in \lCP_+$. 
Elements of $\lCP_+$ are called {\it dominant $\ell$-weights}.

One checks that $\varpi(A_{i,x}) = \alpha_i$ and $\varpi(Y_{i,x}) = \varpi_i$. So $\varpi(\lCQ_{-}) = \BQ_{-}$.

Consider an irreducible module $S := L(\Bd)$ in category $\BGG$. The {\it height} of
an $\ell$-weight $\Bf$ of $S$, defined as $h(\varpi(\Bf^{-1}\Bd))$, is a non-negative integer since $\varpi(\Bd^{-1}\Bf) \in \BQ_-$. For $n \in \BZ_{\geq 0}$, let $\lwt^n(S)$ be the set of $\ell$-weights of $S$ of height $n$. Then $\lwt(S)$ is a disjoint union of these $\lwt^n(S)$. In particular, $\{\Bd\} = \lwt^0(S)$.

\begin{cor} \label{cor: layer structure of l-weights}
Let $S$ be an irreducible module in category $\BGG$. We have 
$$ \lwt^{n+1}(S) \subseteq  \lwt^n(S) \{ A_{j,x}^{-1}\ |\ (j,x) \in I \times \BC \} \quad \mathrm{for}\ n \in \BZ_{\geq 0}.  $$
In particular, $\lwt(S) \subseteq \Bd \lCQ_-$ where $\Bd$ is the highest $\ell$-weight of $S$.
\end{cor}
\begin{proof}
Let $\Be \in \lwt^{n+1}(S)$. Since $\Be \neq \Bd$,  the $\ell$-weight space $S_{\Be}$ is not the highest $\ell$-weight space. There exists $j \in I$ such that $x_j^+(u) S_{\Be} \neq 0$. Such a non-zero space must project non-trivially to certain $\ell$-weight space  $S_{\Bf}$. By Proposition \ref{prop: generalized simple roots}, there exists $x \in \BC$ such that $\Bf = A_{j,x} \Be$. We have $\Be = A_{j,x}^{-1} \Bf$ and $\Bf \in \lwt^{n}(S)$. 
\end{proof}  
Let $V$ be a module in category $\BGG$ which admits a one-dimensional weight space $V_{\lambda}$ such that $\wt(V) \subset \lambda + \BQ_-$. Then $V_{\lambda}$ is also an $\ell$-weight space of $\ell$-weight $\Bd$, and we define the {\it normalized q-character} by
$ \nqc(V) := \qc(V) \times \Bd^{-1} \in \CEl$.
If $V$ is irreducible, then Corollary \ref{cor: layer structure of l-weights} implies that $\nqc(V)$ is a (possibly infinite) sum of monomials in the $A_{j,x}^{-1}$ with leading term $1$. 
 
We introduce the {\it completed Grothendieck ring} $K_0(\BGG)$ as in \cite[\S 3.2]{HL}. Its elements are formal sums $\sum_{\Bd \in \CR'} n_{\Bd} [L(\Bd)]$ of the symbols $[L(\Bd)]$ with integer coefficients $n_{\Bd}$ such that the direct sum of $\Yang$-modules $\oplus_{\Bd \in \CR'} L(\Bd)^{\oplus |n_{\Bd}|}$ is in category $\BGG$. Addition is the usual one of formal sums. 

Let $V$ be in category $\BGG$. In general $V$ may not admit a Jordan--H\"older series of finite length. Still, as in the case of Kac--Moody algebras \cite[\S 9.3]{Kac}, for $\Bd \in \CR'$ the multiplicity $m_{\Bd,V} \in \BZ_{\geq 0}$ of the irreducible module $L(\Bd)$ in  $V$ makes sense, and $[V] := \sum_{\Bd \in \CR'} m_{\Bd,V} [L(\Bd)]$ is a well-defined element in $K_0(\BGG)$.  Multiplication in $K_0(\BGG)$ is induced by $[V][W] = [V \otimes W]$ for $V, W$ in category $\BGG$. 

 Since $\qc$ respects exact sequences, the assignment $[V] \mapsto \qc(V)$ extends uniquely to a group homomorphism  $\qc: K_0(\BGG) \longrightarrow \CEl$, called the q-character map.

\begin{theorem}\cite{Kn}  \label{thm: q-char ring morphism}
The q-character map is an injective homomorphism of rings. 
\end{theorem}
We conclude as in \cite[Remark 3.13]{HJ} that $K_0(\BGG)$ is a commutative ring. 

\begin{example}  \label{example: KR sl2}
Take $\Glie = \mathfrak{sl}_2$, so that $I = \{1\}$ and $c_{11} = 2 = 2d_1$. For $k,x \in \BC$, there is a representation $\SL^{x+k}_x$ of $Y_{\hbar}(\mathfrak{sl}_2)$ on $\oplus_{i=0}^{\infty} \BC v_i$ defined by \cite[Proposition 2.6]{CP}:
    \begin{gather*}
              x_{1,n}^+ v_i = (-x + 1 - i)^n \hbar^n v_{i-1}, \quad x_{1,n}^- v_i = (-x-i)^n\hbar^n (i+1)(k-i) v_{i+1},  \\
              \xi_1(u) v_i = \frac{(u+(x-1)\hbar)(u+(x+k)\hbar)}{(u+(x+i-1)\hbar)(u+(x+i)\hbar)} v_i.                   
    \end{gather*}
It is in category $\BGG$ with q-character $\qc(\SL_x^{x+k}) = \frac{\Psi_{1,x+k}}{\Psi_{1,x}} \qc(\SL_x^x)$ and
$$ \qc(\SL^{x}_x) = 1 + A_{1,x}^{-1} + A_{1,x}^{-1}A_{1,x+1}^{-1}  + A_{1,x}^{-1}A_{1,x+1}^{-1} A_{1,x+2}^{-1} + \cdots. $$
If $k \in \BZ_{\geq 0}$, then  $v_0, v_1, \cdots, v_k$ span a submodule $\BC_x^{k+1} \cong L(\frac{\Psi_{1,x+k}}{\Psi_{1,x}})$. Comparing q-characters gives $[\BC_x^2] [\SL_y^x] = [\SL_y^{x+1}] + [\SL_y^{x-1}]$ for $y \in \BC$. Restricted to $\mathfrak{sl}_2$ this is Equation \eqref{equ: sl2}. Indeed, $\BC_x^2 \cong \BC^2$, and $\SL_y^x$ is the graded dual of the Verma module $M_{x-y}$ twisted by Cartan involution so that their isomorphism classes coincide.
\end{example}

\section{Kirillov--Reshetikhin modules} \label{sec: KR}
In this section we study two families of finite-dimensional irreducible modules and their q-character properties. For
 $(i, k, x) \in I \times \BZ_{\geq 0} \times \BC$ define
\begin{equation} \label{equ: KR dominant}
\Bw_{k,x}^{(i)} := Y_{i,x+\frac{1}{2}d_i} Y_{i,x+\frac{3}{2}d_i} \cdots Y_{i,x+(k-\frac{1}{2})d_i} = \frac{\Psi_{i,x+kd_i}}{\Psi_{i,x}} \in \lCP_+.
\end{equation}
The finite-dimensional irreducible module $L(\Bw_{k,x}^{(i)})$ is denoted by $W_{k,x}^{(i)}$ and called {\it Kirillov--Reshetikhin module} \cite{KR}, KR module for short. 

Let $\lCQ_x^-$ be the submonoid of $\lCQ_-$ generated by the $A_{j,w}^{-1}$ for $j \in I$ and $w \in x+\frac{1}{2}\BZ$. 

\begin{theorem} \cite{H1, Nakajima}  \label{thm: l-weights KR}
Let $(i, x) \in I \times \BC$. For $k \in \BZ_{>0}$ and $0 \leq l < k$:
\begin{itemize}
\item[(a)] $\Bw_{k,x}^{(i)}A_{i,x}^{-1} A_{i,x+d_i}^{-1}  \cdots A_{i,x+ld_i}^{-1}$ is an $\ell$-weight of $W_{k,x}^{(i)}$ of multiplicity one; 
\item[(b)] any $\ell$-weight of $W_{k,x}^{(i)}$ different from (a) and from $\Bw_{k,x}^{(i)}$ must be of the form $\Bw_{k,x}^{(i)} A_{i,x}^{-1} A_{i',z}^{-1} \Be$ where $\Be \in \lCQ_x^-$, $i' \in I \setminus \{i\}$, and $z$ belongs to the subset $\mathfrak{Z}_{k,x}^{(ii')}$ of $x + \{-\frac{3}{2}, -1, -\frac{1}{2}, 0, \frac{1}{2} \}$ defined at the second column of the table:
\begin{equation} \label{tab: denominator}
\begin{tabular}{|c|c|}
  \hline 
    conditions on $(i', k)$ & the set $\mathfrak{Z}_{k,x}^{(ii')}$  \\
  \hline    
    $c_{ii'} = 0$  & $\emptyset$  \\
    \hline
    $c_{ii'} = -1$ or $k = 1$ & $\{x-d_{ii'}\}$  \\
    \hline
    $c_{ii'} = -2$ and $k > 1$ & $\{x-1, x\}$ \\
    \hline 
    $c_{ii'} = -3$ and $k = 2$ & $\{x-\frac{3}{2}, x-\frac{1}{2}\}$  \\
    \hline
    $c_{ii'} = -3$ and $k > 2$ & $\{x-\frac{3}{2}, x-\frac{1}{2}, x+\frac{1}{2}\}$  \\
  \hline    
\end{tabular}
\end{equation}
\end{itemize}
As $k$ tends to infinity, $\nqc(W_{k,x}^{(i)})$ converges to a power series of the $A_{j,w}^{-1} \in \lCQ_x^-$.
\end{theorem} 
This is a translation of q-character property of KR modules over the quantum affine algebra $U_q(\widehat{\Glie})$ due to H. Nakajima \cite{Nakajima} and D. Hernandez \cite{H1}. Here $q = e^{\pi \iota \hbar}$ and $\hbar \notin \mathbb{Q}$. The irrationality assumption is inessential for Yangians since the Hopf algebras $Y_{\hbar}(\Glie)$ and $Y_{1}(\Glie)$ are isomorphic for all $\hbar \in \BC^{\times}$. One applies the inverse functor of Gautam--Toledano Laredo \cite[\S 6]{GTL}, which sends finite-dimensional irreducible $U_{q}(\widehat{\Glie})$-modules to irreducible $\Yang$-modules, and matches the q-character maps of Frenkel--Reshetikhin \cite{FR} and H.~Knight \cite{Kn}. Let $\mathcal{Y}_{i,a},\ \mathcal{A}_{i,a}$ denote the monomials $Y_{i,a},\ A_{i,a}$ for $U_q(\widehat{\Glie})$ in \cite[\S 2.3.1]{H1}. At the level of q-characters, the ring morphism $e_{\Pi}$ of \cite[\S 7.7]{GTL} is determined by: \footnote{This is a comparison of rational functions $\frac{u-a\hbar+d_i\hbar}{u-a\hbar}$ and $q_i^{-1}\frac{q_i^2 z^{-1} - q^{2a}}{z^{-1} - q^{2a}}$ from \cite[\S 1.10]{GTL}. We take $z^{-1}$ because of different conventions of $\Psi_i(z)^{\pm}$ in \cite[(1.2)]{GTL} and \cite[\S 2.3.1]{H1}.}
\begin{equation} \label{equ: GTL}
e_{\Pi}(Y_{i,a}) = \mathcal{Y}_{i,q^{-2a}},\quad e_{\Pi}(A_{i,a}) = \mathcal{A}_{i,q^{-2a}} \quad \mathrm{for}\ (i,a) \in I \times \BC. 
\end{equation}
Our condition $w \in x + \frac{1}{2} \BZ$ translates as $q^{-2w} \in  q^{-2x + \BZ}$ in \cite[Theorem 3.4 (2)]{H1}.

Part (a) of Theorem \ref{thm: l-weights KR} generalizes Example \ref{example: KR sl2}. For part (b), let $1 \leq m \leq k$. There is a unique decomposition of $\ell$-weights in terms of ratios of $\Psi$:
\begin{align*}
 &\Bw_{k,x}^{(i)}  A_{i,x}^{-1} A_{i,x+d_i}^{-1} \cdots A_{i,x+(m-1)d_i}^{-1}  \\
  = & \frac{\Psi_{i,x-d_i} \Psi_{x+kd_i}}{\Psi_{i,x+(m-1)d_i} \Psi_{i,x+md_i}} \times \prod_{i': c_{ii'} < 0} \prod_{z \in \mathfrak{Z}_{m,x}^{(ii')}}  \frac{\Psi_{i',z+m_z d_i d_{i'}}}{\Psi_{i',z}}
\end{align*}
where each $m_z$ is an integer such that $1 \leq m_z d_{i'} \leq m$. The presence of $\Psi_{i',z}$ in the denominator gives rise to the factor $A_{i',z}^{-1}$ in part (b). See \cite[Lemma 5.5]{H1}.

We recall an important notion due to Frenkel--Mukhin \cite[\S 6]{FM}.
Since the abelian group $\lCP$ is freely generated by the $Y_{j,x}$, each element $\Bm \in \lCP$ is factorized uniquely as a finite product
$\prod_{(i,x) \in I \times \BC} Y_{i,x}^{c_{i,x}(\Bm)}$ where the exponents $ c_{i,x}(\Bm)$ are integers. Call $\Bm$ {\it right-negative} if $\Bm \neq 1$ and 
      \begin{itemize}
         \item[(RN)] if $c_{i,x}(\Bm) >0$ for certain $(i,x) \in I \times \BC$, then there exists $(j,y) \in I \times \BC$ such that $c_{j,y}(\Bm) < 0$ and $y - x \in \frac{1}{2}\BZ_{<0}$.
      \end{itemize}
      A right-negative $\ell$-weight is never dominant. The product of two (and hence finitely many) right-negative $\ell$-weights is still right-negative.

\begin{example} \cite{FM} \label{example: simple roots right-negative}
The inverse of a generalized simple root is right-negative:
$$ A_{j,x}^{-1} = Y_{j,x-\frac{1}{2}d_j}^{-1} Y_{j,x+\frac{1}{2}d_j}^{-1}  \prod_{i:c_{ij}=-1} Y_{i,x} \prod_{i:c_{ij}=-2} Y_{i,x+\frac{1}{2}}Y_{i,x-\frac{1}{2}}\prod_{i:c_{ij}=-3} Y_{i,x+1}Y_{i,x}Y_{i,x-1}. $$
Consequently, if $\Bm \in \lCP$ is right-negative, then so is any element of $\Bm \lCQ_-$.
\end{example}

Now we discuss a second family of finite-dimensional irreducible modules. For $(i, x, k, t) \in I \times \BC \times \BZ_{>0} \times \BZ_{\geq 0}$, the following is a dominant $\ell$-weight:
\begin{equation}  \label{def: Demazure}
\Bd_{k,x}^{(i,t)} = \Bw_{k,x-(k+1)d_i}^{(i)} \Bw_{k+t,x-kd_i}^{(i)} A_{i,x-d_i}^{-1} A_{i,x-2d_i}^{-1} \cdots A_{i,x-kd_i}^{-1}.
\end{equation}
The finite-dimensional irreducible module $L(\Bd_{k,x}^{(i,t)})$ is denoted by $D_{k,x}^{(i,t)}$.
 
\begin{theorem}  \label{thm: KR and Demazure}
For $(i, k, t) \in I \times \BZ_{>0} \times  \BZ_{\geq 0}$, there exists (up to scalar multiple) a unique short exact sequence of finite-dimensional $Y_{\hbar}(\Glie)$-modules
\begin{equation*} 
0 \longrightarrow D_{k,(k+1)d_i}^{(i,t)} \longrightarrow  W_{k,0}^{(i)} \otimes W_{k+t,d_i}^{(i)} \longrightarrow W_{k-1,d_i}^{(i)} \otimes W_{k+t+1,0}^{(i)} \longrightarrow 0.
\end{equation*}
\end{theorem}
\begin{proof}
Let us introduce the tensor products of $Y_{\hbar}(\Glie)$-modules
$$ V :=  W_{k-1,d_i}^{(i)} \otimes W_{k+t+1,0}^{(i)},\quad V' := W_{k,0}^{(i)} \otimes W_{k+t,d_i}^{(i)}.  $$
They share a common $\ell$-weight $ \Bw_{k-1,d_i}^{(i)} \Bw_{k+t+1,0}^{(i)} = \Bw_{k,0}^{(i)} \Bw_{k+t,d_i}^{(i)}$, denoted by $\Bw$.

{\bf Step 1.} We count the dominant $\ell$-weights in $V$ and $V'$. A lengthy but straightforward calculation based on Table \eqref{tab: denominator} and Example \ref{example: simple roots right-negative} indicates that the following $\ell$-weights are always right-negative (compare with \cite[Lemma 5.2]{H1}):
$$\Bw_{k,0}^{(i)} A_{i,0}^{-1} \Bw_{k+t,d_i}^{(i)},\quad \Bw_{k,0}^{(i)} \Bw_{k+t,d_i}^{(i)} A_{i,d_i}^{-1} A_{i',z}^{-1}\quad \mathrm{for}\ i' \in I \setminus \{i\}\ \mathrm{and}\ z \in \mathfrak{Z}_{k+t,d_i}^{(ii')}.$$
Now the arguments of \cite[Lemma 5.6]{H1} work here. All dominant $\ell$-weights of $V$ and of $V'$ are multiplicity-free.  Those of $V$ are listed below:
$$ \Bw,\quad  \Bw A_{i,d_i}^{-1},\quad  \Bw A_{i,d_i}^{-1} A_{i,2d_i}^{-1}, \cdots,\quad \Bw A_{i,d_i}^{-1} A_{i,2d_i}^{-1}\cdots A_{i,(k-1)d_i}^{-1}.  $$
Dominant $\ell$-weights of $V'$ are the above ones together with $\Bd_{k,(k+1)d_i}^{(i,t)}$.
 
We show that $V$ is irreducible and $V'$ projects onto $V$. The irreducibility would force $V'$ to have precisely two irreducible sub-quotients, while the projection would be completed to the desired short exact sequence. 

Let $Y_i$ denote the subalgebra of $\Yang$ generated by the $x_{i,n}^{\pm}, \xi_{i,n}$ for $n \in \BZ_{\geq 0}$. As an algebra it is naturally isomorphic to the Yangian $Y_{d_i\hbar}(\mathfrak{sl}_2)$, so the coproduct of the latter induces a coproduct $\Delta_i$ of $Y_i$. Combing \cite[(1.3)]{CP} with \cite{Kn}, we have
 $$\Delta(x_{i,0}^{\pm}) = \Delta_i(x_{i,0}^{\pm}),\quad  \Delta(\xi_{i,1}) - \Delta_i(\xi_{i,1}) \in \sum_{\alpha \succ 0, \alpha \neq \alpha_i} \Glie_{-\alpha} \otimes \Glie_{\alpha}, $$
the summation over  positive roots of $\Glie$ distinct from $\alpha_i$.
 
{\bf Step 2.} To show that $V$ is irreducible, let $S$ be the irreducible sub-quotient of $V$ of highest $\ell$-weight $\Bw$. It suffices to show that all the dominant $\ell$-weights of $V$ appear in $S$. By Step 1, it is enough to show that for each $n \in \BZ_{>0}$ the weight spaces $S_{\varpi(\Bw)-n\alpha_i}$ and $V_{\varpi(\Bw)-n\alpha_i}$  are of the same dimension. Note that the former is already bounded above by the latter. We are reduce to the inequality
$$ \dim S_{\varpi(\Bw) - n \alpha_i} \geq \dim V_{\varpi(\Bw)-n\alpha_i}. $$
%By the tensor product decomposition of $V$ and Theorem \ref{thm: l-weights KR}, $\dim V_{\varpi(\Bw)-n\alpha_i}$ is the number of couples $(n_1, n_2)$ such that $0 \leq n_1 \leq k - 1,\ 0 \leq n_2 \leq k + t + 1$ and $n = n_1 + n_2$. ($n_1$ comes from $W_{k-1,d_i}^{(0)}$ and $n_2$ from $W_{k+t+1,0}^{(i)}$.) Namely
%$$ \dim V_{\varpi(\Bw)-n\alpha_i} = \delta(n \leq 2k+t) \times \min\{n+1, k\}. $$
%Here for $P$ a statement, $\delta(P) = 1$ if $P$ is true and  $\delta(P) = 0$ if $P$ is false.

Fix a highest $\ell$-weight vector $\omega$ of $S$. The subspace $Y_i \omega$, as a $Y_{d_i\hbar}(\mathfrak{sl}_2)$-module, is of highest $\ell$-weight 
$$ \frac{u+k d_i\hbar}{u+d_i\hbar}  \times \frac{u + (k+t+1)d_i\hbar}{u}. $$
The two components $\frac{u+k d_i\hbar}{u+d_i\hbar} $ and $ \frac{u + (k+t+1)d_i\hbar}{u}$, viewed as highest $\ell$-weights for the Yangian $Y_{d_i\hbar}(\mathfrak{sl}_2)$ and then identified with strings $S_{k-1}(-\frac{1}{2}k)$ and $S_{k+t+1}(-\frac{1}{2}k-\frac{1}{2}t)$ via \cite[Corollary 2.7]{CP}, are non-interacting in the sense of \cite[Definition 3.2]{CP}; indeed, the first string is contained in the second. By \cite[Theorem 4.1]{CP}, the tensor product of irreducible modules attached to these two strings is irreducible, so the $Y_{d_i\hbar}(\mathfrak{sl}_2)$-module $Y_i \omega$ projects onto $L(\frac{ \Psi_{1,k}}{\Psi_{1,1}}) \otimes L( \frac{\Psi_{1,k+t+1}}{\Psi_{1,0}})$ in the notations of Example \ref{example: KR sl2}. The weight spaces $S_{\varpi(\Bw)-n\alpha_i}$ and $(Y_i \omega)_{(2k+t)\varpi_1-n\alpha_1}$ are equal as they are both  spanned by the $x_{i,k_1}^- x_{i,k_2}^- \cdots x_{i,k_n}^- \omega$ for $k_1,k_2, \cdots, k_n \in \BZ_{\geq 0}$. Therefore,
\begin{align*}
& \dim S_{\varpi(\Bw)-n\alpha_i} = \dim (Y_i \omega)_{(2k+t)\varpi_1-n\alpha_1}  \\
\geq &  (L(\frac{ \Psi_{1,k}}{\Psi_{1,1}}) \otimes L( \frac{\Psi_{1,k+t+1}}{\Psi_{1,0}}))_{(2k+t)\varpi_1-n\alpha_1}  = \dim V_{\varpi(\Bw)-n\alpha_i}.
\end{align*}
The last identity comes from the tensor product factorization of $V$ and Theorem \ref{thm: l-weights KR}.

{\bf Step 3.} Let $\omega_1$ and $\omega_2$ be highest $\ell$-weight vectors of $W_{k,0}^{(i)}$ and $W_{k+t,d_i}^{(i)}$ respectively. Then the submodule $Z$  of $V'$ generated by the highest $\ell$-weight vector $\omega_1 \otimes \omega_2$ projects onto $V$. We prove that for each $n \in \BZ_{>0}$ the weight spaces $Z_{\varpi(\Bw)-n\alpha_i}$ and $V_{\varpi(\Bw)-n\alpha_i}'$ are of the same dimension. This would imply that $Z$ and $V$ are non-isomorphic (as the dimensions of the corresponding weight spaces for $n = k$ differ), $Z$ admits at least two irreducible sub-quotients, and so $Z = V'$. 

Let $\alpha$ be a positive root of $\Glie$ different from $\alpha_i$.  The subspace $\Glie_{\alpha} Y_i \omega_2$, if non-zero, would imply that $\alpha - m \alpha_i + (k+t)\varpi_i$ is a weight of $W_{k+t,d_i}^{(i)}$ for certain $m \in \BZ$ and so it must be bounded above by $(k+t)\varpi_i$, meaning $\alpha - m\alpha_i$ is a sum of negative roots of $\Glie$, impossible. Therefore $\Glie_{\alpha} Y_i \omega_2 = 0$ and
 $$ \Delta(x) (v_1 \otimes v_2) = \Delta_i(x) (v_1 \otimes v_2) \in Y_i \omega_1 \otimes Y_i \omega_2 \subset Z $$
 for $x \in \{x_{i,0}^{\pm}, \xi_{i,1} \}$ and $v_1 \otimes v_2 \in Y_i \omega_1 \otimes Y_i \omega_2$. Since the algebra $Y_i$ is generated by $x_{i,0}^{\pm}$ and $\xi_{i,1}$, the above identity holds true for all $x \in Y_i$. In other words, $Y_i \omega_1 \otimes Y_i \omega_2$ is actually a sub-$Y_i$-module of $Z$ and it is identified with the ordinary tensor product $Y_i \omega_1 \otimes Y_i \omega_2$ of $Y_{d_i\hbar}(\mathfrak{sl}_2)$-modules defined by $\Delta_i$. 
 
 By Theorem \ref{thm: l-weights KR}, the $Y_{d_i\hbar}(\mathfrak{sl}_2)$-modules $Y_i \omega_1$ and $Y_i \omega_2$ are isomorphic to $L(\frac{\Psi_{1,k}}{\Psi_{1,0}})$ and $L(\frac{\Psi_{1,k+t+1}}{\Psi_{1,1}})$ respectively.
 To prove that dimensions match, as in Step 2 it suffices to establish that the $Y_{d_i\hbar}(\mathfrak{sl}_2)$-module $L(\frac{\Psi_{1,k}}{\Psi_{1,0}}) \otimes L(\frac{\Psi_{1,k+t+1}}{\Psi_{1,1}})$ is of highest $\ell$-weight. Such a tensor product corresponds to 
 $$W_{k}(-\frac{1}{2}k+\frac{1}{2}) \otimes W_{k+t}(-\frac{1}{2}k-\frac{1}{2}t)$$
  in \cite[Corollary 2.7]{CP}, and is of highest $\ell$-weight by \cite[Proposition 4.2]{CP}.
\end{proof}
We have followed the general scheme of the proof of T-systems for quantum affine algebras \cite{H1,Nakajima}, which corresponds to the case $t = 0$. The differences lie at Steps 2 and 3, which are variants of Chari's reduction trick \cite{Chari}. 

 For $t > 0$, we note that the two tensor products $V$ and $V'$ appeared in the proof of \cite[Theorem 4.1]{FoH} with $m_1 = k+t+1$ and $m_2 = k-1$ (so that $m_1 > m_2 + 2$). Actually \cite{FoH} established that $V$ is an irreducible sub-quotient of $V'$ without mentioning the other irreducible sub-quotient. That $V'$ is a highest $\ell$-weight module can be deduced from more general results about tensor products of KR modules \cite{TG,T}.  
 \begin{cor} \label{cor: l-weight Demazure}
 Let $(i, k, x) \in I \times \BZ_{>0} \times \BC$. If $\Bd_{k,x}^{(i,1)} \Be$ is an $\ell$-weight of $D_{k,x}^{(i,1)}$, then either $\Be \in \{1, A_{i,x}^{-1}\}$ or $\Be \in A_{i',x-kd_i+z}^{-1} \lCQ_x^-$ for certain $(i',z) \in I \times \frac{1}{2}\BZ$ with
 \begin{align*}
\mathrm{either}\quad (i' = i,\ z = -d_i) \quad \mathrm{or}\quad (c_{ii'} < 0,\ -\frac{3}{2} \leq z \leq \frac{1}{2}).
 \end{align*}
\end{cor}
\begin{proof}
We have the pullback formula $\tau_a^*(D_{k,x}^{(i,t)}) = D_{k,x+a}^{(i,t)}$ by Remark \ref{rem: spectral shift}.
Assume without loss of generality $x = kd_i$. By Theorem \ref{thm: KR and Demazure}, $D_{k,kd_i}^{(i,1)}$ is an irreducible sub-quotient of $W_{k,-d_i}^{(i)} \otimes W_{k+1,0}^{(i)}$. It follows from Equation \eqref{def: Demazure} that
$$ A_{i, (k-1)d_i}^{-1} A_{i,(k-2)d_i}^{-1} \cdots A_{i,d_i}^{-1} A_{i,0}^{-1} \Be = \Be' \Be''$$ 
where $\Be'$ and $\Be''$ are monomials in the normalized q-characters of the first and the second KR modules respectively. Note that $\Be, \Be', \Be'' \in \lCQ_0^-$. Suppose $\Be \neq 1$.

 If $\Be' \neq 1$, then $\Be' \in A_{i,-d_i}^{-1} \lCQ_0^-$, which implies $\Be \in A_{i,-d_i}^{-1} \lCQ_0^-$. 
 
 Assume $\Be' = 1$. Then $\Be'' \in A_{i, (k-1)d_i}^{-1} A_{i,(k-2)d_i}^{-1} \cdots A_{i,d_i}^{-1} A_{i,0}^{-1} \lCQ_0^-$. Applying Theorem \ref{thm: l-weights KR} to the KR module $W_{k+1,0}^{(i)}$, we have: 
\begin{itemize}
\item either $\Be'' = A_{i, (k-1)d_i}^{-1} A_{i,(k-2)d_i}^{-1} \cdots A_{i,d_i}^{-1} A_{i,0}^{-1} A_{i,kd_i}^{-1}$, implying $\Be = A_{i,kd_i}^{-1}$;
\item or $\Be'' \in  A_{i',z}^{-1} \lCQ_0^-$ where $c_{ii'} < 0$ and $z \in \mathfrak{Z}_{k+1,0}^{(ii')} \subseteq \{-\frac{3}{2},\ -1,\ -\frac{1}{2},\ 0,\ \frac{1}{2} \}$, implying $\Be \in A_{i',z}^{-1} \lCQ_0^-$ for such a pair $(i', z)$.
\end{itemize} 
The proof of the corollary is completed.
 \end{proof}

Let us write $\Bd_{k,x}^{(i,t)}$ in terms of the $\Psi$ based on Equation \eqref{def: simple root}. Assume $k \in 6 \BZ_{>0}$ so that $k$ is divisible by all the $d_{j}$. We have
\begin{multline}  \label{equ: Demazure weight}
\quad \quad \quad \Bd_{k,x}^{(i,t)} = \frac{\Psi_{i,x+td_i}}{\Psi_{i,x}} \prod_{j: c_{ij} < 0} \frac{\Psi_{j,x+d_{ij}}}{\Psi_{j,x + d_{ij}-k d_i}}  \\
  \times  \prod_{j: c_{ij} = -2} \frac{\Psi_{j,x}}{\Psi_{j,x-k} }  \times \prod_{j: c_{ij} = -3} \frac{\Psi_{j,x+\frac{1}{2}} \Psi_{j,x-\frac{1}{2}}}{\Psi_{j,x+\frac{1}{2}-k} \Psi_{j,x-\frac{1}{2}-k}}.
\end{multline}
One may view $\Bd_{k,0}^{(i,t)}$ as the Yangian analog of $\tilde{\Psi}_i^{(t,k)}$ in \cite[\S 4.3]{FH2}.
If $\Glie$ is simply-laced, then the second line disappears and Corollary \ref{cor: l-weight prime Demazure} will follow directly.

\section{Asymptotic representations} \label{sec: asym}
We construct infinite-dimensional asymptotic modules in category $\BGG$ as limits of KR modules and provide three-term identities in the Grothendieck ring. 

\subsection{Asymptotic construction.} \label{sub: idea}
  Let us recall the asymptotic construction from \cite[Section 2]{Z2}. Suppose an algebra $A$ (unital and associative) is given. Let $S$ be a set of algebraic generators for $A$.

For all positive integer $k \in \BZ_{>0}$ let a representation $\rho_k: A \longrightarrow \mathrm{End}_{\BC}(V_k)$ of $A$ be given. Let $(F_{k,l}: V_l \longrightarrow V_k)_{l<k}$ be an inductive system of vector spaces; namely $F_{k,l} F_{l,m} = F_{k,m}$ for $m<l<k$ as linear maps $V_m \longrightarrow V_k$. 

\noindent {\bf Assumptions.} We suppose that the structural maps $F_{k,l}$ are {\it injective}. Next, we assume the $(L,K)$-{\it asymptotic property} for fixed $L, K \in \BZ_{>0}$: 
\begin{enumerate}
\item[(AP)] for $(s, l) \in S \times \BZ_{>0}$, there exists a $\mathrm{Hom}_{\BC}(V_l,V_{l+L})$-valued polynomial $P_{s;l}(u)$ in $u$ of degree bounded by $K$ such that
$$ s F_{k,l} = F_{k,l+L} P_{s;l}(u)|_{u = k} \in \mathrm{Hom}_{\BC}(V_l,V_k) \quad \mathrm{for}\ k > l+L. $$
\end{enumerate}
Here to be precise one should write $\rho_k(s) F_{k,l}$ instead of $s F_{k,l}$. We omit $\rho_k$ because it is, and will be, clear from the context to which representation we are applying $s$.
As in \cite{Z2}, the polynomial $P_{s;l}(u)$ is unique. For fixed $s \in S$, the coefficients of the polynomials $(P_{s;l}(u))_{l>0}$ form morphisms of inductive systems: 
$$F_{l+L,m+L} P_{s;m}(u) = P_{s;l}(u) F_{l,m} \in \mathrm{Hom}_{\BC}(V_m, V_{l+L})[u] \quad \mathrm{for}\ m < l. $$ 
Their inductive limits are encoded in an $\mathrm{End}_{\BC} (V_{\infty})$-valued polynomial $P_s(u)$ of degree bounded by $K$. Here $V_{\infty}$ denotes the inductive limit of $(V_l, F_{k,l})$.

\noindent {\bf Claim.} Let $c \in \BC$. Then $s \mapsto P_s(u)|_{u=c}$ defines a representation of $A$ on $V_{\infty}$.

Later in Section \ref{sec: prefund}, we shall be in the situation where $P_{s;l}(u)$ are polynomials in $u^{-1}$ of bounded degree. In this case, one can take $c = \infty$ to obtain a representation. 

Our construction is inspired by the work of Hernandez--Jimbo \cite{HJ} on representations of an upper Borel subalgebra of the quantum affine algebra $U_q(\widehat{\Glie})$. For the Borel subalgebra the asymptotic property is described by polynomials in $q^{-k}$, and Hernandez--Jimbo took the specialization of $q^{-k}$ at zero to define the negative pre-fundamental modules. For the full quantum affine algebra the asymptotic property involves Laurent polynomials in $q^{k}$, and one has specializations of $q^{k}$ at non-zero complex numbers, resulting in asymptotic modules \cite{Z2}.

\subsection{Inductive system of KR modules.}
From now on up to Definition \ref{defi: asymptotic module}, $i \in I$ is fixed, and when writing $k > l$ we implicitly assume that $k,l$ are non-negative integers. We shall produce an inductive system of vector spaces $W_{0,0}^{(i)} \subset W_{1,0}^{(i)} \subset W_{2,0}^{(i)} \subset \cdots $, and then establish asymptotic property for this system. 

For $k > l$ let $Z_{lk}$ be the KR module $W_{k-l,ld_i}^{(i)}$ with a fixed highest $\ell$-weight vector $\omega_{lk}$. (This differs from the notation $Z_{kl}$ in \cite[\S 4]{Z4} due to the opposite tensor products taken in Theorem \ref{thm: KR and Demazure} and in  \cite[Lemma 3.6]{Z4}.) In the case $l = 0$ we have $Z_{0k} = W_{k,0}^{(i)}$ and write simply $\omega_k := \omega_{0k}$.

Setting $k=1$ in Theorem \ref{thm: KR and Demazure}, we get surjective morphisms of $\Yang$-modules $W_{1,0}^{(i)} \otimes W_{t+1,d_i}^{(i)} \longrightarrow W_{t+2,d_i}^{(i)}$ for $t \geq 0$ relating different KR modules. The proof of \cite[Lemma 4.1]{Z4} then goes without change. For $m < l < k$ there exists a unique $\Yang$-module morphism 
$ Z_{ml} \otimes Z_{lk} \longrightarrow Z_{mk} $
which sends $\omega_{ml} \otimes \omega_{lk}$ to $\omega_{mk}$. As two special cases, for $l<k$ and $l < t-1$ we have 
$$ \SF_{k,l}: W_{l,0}^{(i)} \otimes Z_{lk} \longrightarrow W_{k,0}^{(i)},\quad \SG_{t,l}: Z_{l,l+1} \otimes Z_{l+1,t} \longrightarrow Z_{lt}. $$
As in \cite[\S 4.2]{HJ}, we identify $W_{l,0}^{(i)}$ with the vector subspace $W_{l,0}^{(i)} \otimes \omega_{lk}$ of $W_{l,0}^{(i)} \otimes Z_{lk}$, and then we restrict $\SF_{k,l}$ to this subspace to get the linear map
$$ F_{k,l}: W_{l,0}^{(i)} \longrightarrow W_{k,0}^{(i)},\quad  v \mapsto \SF_{k,l}(v \otimes \omega_{lk}). $$
Then $F_{k,l}(\omega_l) = \omega_k$, and $(W_{l,0}^{(i)}, F_{k,l})$ forms an inductive system of vector spaces: $F_{k,l} F_{l,m} = F_{k,m}$ for $m < l < k$. If $l < k-1$ then for $w \otimes v \in W_{l,0}^{(i)} \otimes Z_{l,l+1}$ we have:
\begin{gather}
 \SF_{k,l}(w \otimes \SG_{k,l}(v \otimes \omega_{l+1,k})) =  F_{k,l+1} \SF_{l+1,l}(w \otimes v).  \label{rel: F G}
\end{gather}
(There is a typo at the right-hand side of \cite[Eq.(4.19)]{Z4}: $F_{k,l}$ should be $F_{k,l+1}$.)
\begin{lem}  \label{lem: compt with x+}
For $l < k$ the map $F_{k,l}$ is an injective morphism of $Y_{\hbar}^+(\Glie)$-modules. 
\end{lem}
\begin{proof}
Since $\omega_{lk}$ is a highest $\ell$-weight vector, by Remark \ref{rem: q mult}, in $W_{l,0}^{(i)} \otimes Z_{lk}$:
$$ x_{j}^+(u) (w \otimes \omega_{lk}) = x_{j}^+(u) w \otimes \omega_{lk} \quad \mathrm{for}\ w \in W_{l,0}^{(i)}\ \mathrm{and}\ j \in I. $$
It follows from the $\Yang$-linearity of $\SF_{k,l}$ that 
\begin{align*}
x_{j}^+(u) F_{k,l} (w) &= x_{j}^+(u) \SF_{k,l}(w \otimes \omega_{lk}) = \SF_{k,l}x_j^+(u) (w \otimes \omega_{lk}) \\
& = \SF_{k,l}( x_{j}^+(u)w\otimes \omega_{lk}) = F_{k,l} x_j^+(u) w.
\end{align*}
So $x_j^+(u) F_{k,l} = F_{k,l} x_j^+(u)$ for all $j \in I$, meaning that $F_{k,l}$ is $Y_{\hbar}^+(\Glie)$-linear.  The injectivity is proved similarly as \cite[Proposition 4.3 (2)]{Z2}, with the RLL generators therein replaced by the Drinfeld generators; see also \cite[Lemma 4.2]{Z4}.
\end{proof}

Remark \ref{rem: q mult} also gives the following commutation relations for $l < k$ and $j \in I$:
\begin{align} 
&\xi_j(u) F_{k,l} = \left(\frac{u+kd_i \hbar }{u+ld_i \hbar}\right)^{\delta_{ij}} F_{k,l} \xi_j(u), \label{rel: compt with xi j} \\
& x_j^-(u) F_{k,l} = F_{k,l} x_j^-(u) \quad \mathrm{if}\ j \neq i.  \label{rel: comp with x j}
\end{align}
For the last equation we used $x_j^-(u) \omega_{lk} = 0$ for $j \neq i$, based on Lemma \ref{lem: denominator weight space}. 

\begin{lem} \label{lem: compt x i}
For $l < k-1$ we have
\begin{equation} \label{rel: compt with x i}
x_i^-(u) F_{k,l}(w) = F_{k,l+1} ( A_i^l(u) + B_i^l(u) kd_i ),
\end{equation}
where $A_i^l(u)$ and $B_i^l(u)$ are $\mathrm{Hom}_{\BC}(W_{l,0}^{(i)}, W_{l+1,0}^{(i)})$-valued power series in $u^{-1}$ independent of $k$ and defined by the following formulas. For $w \in W_{l,0}^{(i)}$,
\begin{align*}
A_i^l(u): &\quad w \mapsto \frac{u}{u+ld_i\hbar} F_{l+1,l} x_i^-(u) w - \frac{l\hbar}{u+ld_i\hbar}\SF_{l+1,l}(w \otimes  x_{i,0}^- \omega_{l,l+1}), \\
B_i^l(u): &\quad w \mapsto \frac{\hbar}{u+ld_i\hbar} F_{l+1,l} x_i^-(u) w + \frac{\hbar}{(u+ld_i\hbar)d_i} \SF_{l+1,l}(w \otimes  x_{i,0}^- \omega_{l,l+1}).
\end{align*} 
\end{lem}
\begin{proof}
First we prove the following vector is in the kernel of $\SG_{k,l}$:
 $$v := \omega_{l,l+1} \otimes x_{i,0}^- \omega_{l+1,k} - (k-l-1) x_{i,0}^- \omega_{l,l+1} \otimes \omega_{l+1,k} \in Z_{l,l+1} \otimes Z_{l+1,k}. $$
  One checks that $x_{i,0}^+ v = 0$ based on the identity 
$$x_{i,0}^+ x_{i,0}^- \omega_{lk} = [x_{i,0}^+,x_{i,0}^-] \omega_{lk} = \xi_{i,0} \omega_{lk} = (k-l)d_i \omega_{lk}. $$
Note that $\SG_{k,l}(v)$ is in the weight space $(Z_{lk})_{(k-l)\varpi_i-\alpha_i}$, which by Theorem \ref{thm: l-weights KR}  is spanned by $x_{i,0}^- \omega_{lk}$. Now $x_{i,0}^+ x_{i,0}^- \omega_{lk} \neq 0$ and $x_{i,0}^+ \SG_{k,l}(v) = 0$ force $\SG_{k,l} (v) = 0$.

Applying Lemma \ref{lem: denominator weight space} (b) to $Z_{lk}$ we get $(u+ld_i\hbar)x_i^-(u) \omega_{lk} = \hbar x_{i,0}^- \omega_{lk}$ and
\begin{align*}
x_{i,0}^- \omega_{lk} &= x_{i,0}^- \SG_{k,l}(\omega_{l,l+1} \otimes \omega_{l+1,k}) = \SG_{k,l} x_{i,0}^-(\omega_{l,l+1} \otimes \omega_{l+1,k}) \\
&=  \SG_{k,l}(\omega_{l,l+1} \otimes x_{i,0}^-\omega_{l+1,k} +  x_{i,0}^- \omega_{l,l+1} \otimes \omega_{l+1,k} - v) \\
&=  (k-l) \SG_{k,l}(x_{i,0}^- \omega_{l,l+1} \otimes \omega_{l+1,k}), \\
x_i^-(u) \omega_{lk} &= \frac{(k-l)\hbar}{u+ld_i\hbar} \SG_{k,l}(x_{i,0}^- \omega_{l,l+1} \otimes \omega_{l+1,k}).
\end{align*}

We compute $x_i^-(u) F_{k,l}(w)$ for $w \in W_{l,0}^{(i)}$, as in the proof of Lemma \ref{lem: compt with x+}:
\begin{align*}
x_i^-(u) F_{k,l}(w) &=  \SF_{k,l} x_i^-(u) (w \otimes \omega_{lk}) = \SF_{k,l}(w \otimes x_i^-(u) \omega_{lk} +  x_i^-(u) w \otimes \xi_i(u) \omega_{lk}) \\
&= \frac{(k-l)\hbar}{u+ld_i\hbar} \SF_{k,l}(w \otimes  \SG_{k,l}(x_{i,0}^- \omega_{l,l+1} \otimes \omega_{l+1,k}) ) +  \frac{u+kd_i\hbar}{u+ld_i\hbar} F_{k,l} x_i^-(u) w \\
&= F_{k,l+1}\left( \frac{(k-l)\hbar}{u+ld_i\hbar}  \SF_{l+1,l}(w \otimes  x_{i,0}^- \omega_{l,l+1}) + \frac{u+kd_i\hbar}{u+ld_i\hbar} F_{l+1,l} x_i^-(u) w\right).
\end{align*}
The last row is Equation \eqref{rel: F G} applied to $w \otimes x_{i,0}^- \omega_{l,l+1} \in W_{l,0}^{(i)} \otimes Z_{l,l+1}$.   Equation \eqref{rel: compt with x i} is a reformulation of it.
\end{proof}

 Let $\mathcal{S} := \{x_{j,n}^-,\ \xi_{j,n},\ x_{j,n}^+\ |\ (j, n) \in I \times \BZ_{\geq 0}  \}$ denote the set of Drinfeld generators of $\Yang$. 
 As a summary of Lemma \ref{lem: compt with x+} and Equations \eqref{rel: compt with xi j}--\eqref{rel: compt with x i}, the inductive system $W_{0,0}^{(i)} \subset W_{1,0}^{(i)} \subset W_{2,0}^{(i)} \subset \cdots $ satisfies the $(1,1)$-asymptotic property: to $s \in \mathcal{S}$ and $l \in \BZ_{\geq 0}$ are attached linear maps $A_s^l$ and $B_s^l$ from $W_{l,0}^{(i)}$ to $W_{l+1,0}^{(i)}$ such that 
\begin{equation*}  
s F_{k,l} = F_{k,l+1} (A_s^l + B_s^l \times k d_i) \quad \mathrm{for}\ k > l+1.
\end{equation*}
The $A_s^l$ and $B_s^l$ are unique and form morphisms of inductive systems,
so they admit inductive limits, denoted by $A_s$ and $B_s$ respectively, which are linear operators on the inductive limit $W_{\infty}^{(i)} := \lim\limits_{\rightarrow} W_{l,0}^{(i)}$. Now let us apply Subsection \ref{sub: idea}.

\begin{defi}  \label{defi: asymptotic module} 
For $y \in \BC$ the assignment $s \mapsto A_s + y B_s$ as $s \in \mathcal{S}$ defines a $\Yang$-module structure on $W_{\infty}^{(i)}$, denoted by $\SL(\frac{\Psi_{i,y}}{\Psi_{i,0}})$. More generally, for $x \in \BC$ let $\SL(\frac{\Psi_{i,y}}{\Psi_{i,x}})$ be the pullback of $\SL(\frac{\Psi_{i,y-x}}{\Psi_{i,0}})$ by the spectral parameter shift $\tau_x$ from Equation \eqref{def: spectral shift}, and call it an {\it asymptotic module}. (Even if $y = x$ we still use $\SL(\frac{\Psi_{i,x}}{\Psi_{i,x}})$ to emphasize the dependence on $i$ and $x$.)
\end{defi}

This definition was made first in \cite{MS} for $\Glie = \mathfrak{sl}_2$. 

\begin{theorem}  \label{thm: asymptotic module}
The asymptotic modules are in category $\BGG$ with q-characters 
\begin{equation}  \label{equ: q-char asym}
\qc(\SL(\frac{\Psi_{i,y}}{\Psi_{i,x}})) = \frac{\Psi_{i,y}}{\Psi_{i,x}} \times \lim_{l \rightarrow \infty} \nqc(W_{l,x}^{(i)}) \quad \mathrm{for}\ (i, x, y) \in I \times \BC^2.
\end{equation}
Furthermore, the following statements hold in category $\BGG$.
\begin{itemize}
\item[(a)] The set $\CR'$ of highest $\ell$-weights equals $\CR$. 
\item[(b)] Let $V$ be a finite-dimensional module. In the q-character formula of $V$, replace $\qc(V)$ by $[V]$, write each of its $\ell$-weight as a product of the $\frac{\Psi_{i,y}}{\Psi_{i,x}}$ and replace the $\frac{\Psi_{i,y}}{\Psi_{i,x}}$ by the ratios $\frac{[\SL(\frac{\Psi_{i,y}}{\Psi_{i,0}})]}{[\SL(\frac{\Psi_{i,x}}{\Psi_{i,0}})]}$. Then we get a relation in the fractional field of the Grothendieck ring $K_0(\BGG)$.
\item[(c)] For $(i, a, b, x, y) \in I \times \BC^4$ we have in the Grothendieck ring
\begin{equation}  \label{equ: two-terms}
[\SL(\frac{\Psi_{i,b}}{\Psi_{i,a}})] [\SL(\frac{\Psi_{i,y}}{\Psi_{i,x}})] = [\SL(\frac{\Psi_{i,y}}{\Psi_{i,a}})] [\SL(\frac{\Psi_{i,b}}{\Psi_{i,x}})].
\end{equation}
\end{itemize}
\end{theorem}
\begin{proof}
It suffices to prove Equation \eqref{equ: q-char asym}: its right-hand side converges according to Theorem \ref{thm: l-weights KR} and implies the cone conditions of weights; parts (a)--(c) are then proved in the same way as \cite[Theorem 3.11]{HJ} and \cite[Theorem 4.8]{FH}, by replacing the isomorphism class $[V]$ with q-character $\qc(V)$. Since the q-character map is compatible with the spectral parameter shifts, one may assume $x = 0$.

  Consider the module $\SL(\frac{\Psi_{i,y}}{\Psi_{i,0}})$ in Definition \ref{defi: asymptotic module} with the injective structural maps $F_l: W_{l,0}^{(i)} \longrightarrow \SL(\frac{\Psi_{i,y}}{\Psi_{i,0}})$ for $l \geq 0$. Since the pre-factor at the right-hand side of Equation \eqref{rel: compt with xi j} is a polynomial of $k d_i$, by definition of asymptotic representation, when taking inductive limit $k \rightarrow \infty$ in the equation one replaces $kd_i$ by $y$: 
 $$ \xi_j(u) F_l = (\frac{u+y\hbar}{u+ld_i\hbar})^{\delta_{ij}} F_l \xi_j(u). $$  
So $F_l$ sends an $\ell$-weight vector to another $\ell$-weight vector with $\ell$-weight multiplied by $\frac{\Psi_{i,y}}{\Psi_{i,ld_i}}$. This proves the limit formula \eqref{equ: q-char asym}.
\end{proof}

By Equation \eqref{equ: q-char asym},  it makes sense to talk about the normalized q-characters of tensor products of asymptotic modules and irreducible modules. The normalized q-character of $\SL(\frac{\Psi_{i,y}}{\Psi_{i,x}})$ is independent of $y$. Also, any $\ell$-weight of this module different from $\frac{\Psi_{i,y}}{\Psi_{i,x}}$ must belong to $\frac{\Psi_{i,y}}{\Psi_{i,x}} A_{i,x}^{-1} \lCQ_x^-$.

\subsection{Three-term identities.}
For $(i, k, x) \in I \times \BC^2$, let us take the first line of Equation \eqref{equ: Demazure weight} with $t = 1$ to define 
\begin{align}  
\Bm_{k,x}^{(i)} :=  \frac{\Psi_{i,x+d_i}}{\Psi_{i,x}} \prod_{j: c_{ij} < 0} \frac{\Psi_{j,x+d_{ij}}}{\Psi_{j,x + d_{ij}-k d_i}}, \quad \CM_{k,x}^{(i)} := L(\Bm_{k,x}^{(i)}). \label{def: generic prime demazure} 
\end{align}
If $k \in 6 \BZ_{>0}$, then $\Bm_{k,x}^{(i)}$ is dominant and $\CM_{k,x}^{(i)}$ is finite-dimensional. 

\begin{cor}  \label{cor: l-weight prime Demazure}
Let $k \in 6 \BZ_{>0}$ and $(i,x) \in I \times \BC$. If $\Bm_{k,x}^{(i)} \Be$ is an $\ell$-weight of $\CM_{k,x}^{(i)}$, then either $\Be \in \{1, A_{i,x}^{-1}\}$, or $\Be \in A_{j,x+d_{ij}-kd_i}^{-1} \lCQ_x^-$ for certain $j \in I$ with $c_{ij} < 0$. 
\end{cor}
\begin{proof}
View $\CM_{k,x}^{(i)}$ as an irreducible sub-quotient of $W_{1,x}^{(i)} \otimes M$ where
$$ M := \bigotimes_{j: c_{ij} < 0} W_{\frac{k}{d_j}d_i, x+d_{ij}-kd_i}^{(j)}$$
is a tensor product of KR modules of arbitrary order. Then $\Be = \Be' \Be''$ with $\Be'$ and $\Be''$ being monomials in the normalized q-characters of $W_{1,x}^{(i)}$ and $M$ respectively. So $\Be \in \lCQ_x^-$. Together with Corollary \ref{cor: layer structure of l-weights}, we have $\lwt^n(\CM_{k,x}^{(i)}) \subset \lwt^2(\CM_{k,x}^{(i)}) \lCQ_x^-$ for $n \geq 2$. We may therefore assume  that the $\ell$-weight $\Bm_{k,x}^{(i)} \Be$ of $\CM_{k,x}^{(i)}$ is of height $\leq 2$. 

If $\Be'' \neq 1$, Theorem \ref{thm: l-weights KR} applied to $M$, there exists $j \in I$ such that $c_{ij} < 0$ and  $\Be'' \in A_{j, x+d_{ij}-kd_i}^{-1} \lCQ_x^-$. This implies that $\Be \in A_{j, x+d_{ij}-kd_i}^{-1} \lCQ_x^-$.  

Assume $\Be'' = 1$, so that $\Be = \Be'$ and $\Bw_{1,x}^{(i)} \Be$ is an $\ell$-weight of $W_{1,x}^{(i)}$ of height $\leq 2$.  By Theorem \ref{thm: l-weights KR}, either $\Be \in \{1, A_{i,x}^{-1}\}$, or $\Be = A_{i,x}^{-1} A_{i',x+d_{ii'}}^{-1}$ for certain $i' \in I$ with $c_{ii'} < 0$, because $\mathfrak{Z}_{1,x}^{(ii')} = \{x+d_{ii'}\}$. Let us assume that the latter case happens and arrive at a contradiction.
View $\CM_{k,x}^{(i)}$ as an irreducible sub-quotient of $D_{k,x}^{(i,1)} \otimes A$ where $A$ is the following tensor product of asymptotic modules
$$ A := \bigotimes_{j: c_{ij}=-2} \SL(\frac{\Psi_{j,x-k}}{\Psi_{j,x}}) \otimes \bigotimes_{j:c_{ij}=-3} \left( \SL(\frac{\Psi_{j,x+\frac{1}{2}-k}}{\Psi_{j,x+\frac{1}{2}}}) \otimes \SL(\frac{\Psi_{j,x-\frac{1}{2}-k}}{\Psi_{j,x-\frac{1}{2}}}) \right). $$
Then $A_{i,x}^{-1} A_{i',x+d_{ii'}}^{-1} = \Be^D \Be^A$ where $\Be^D$ and $\Be^A$ are monomials in $\nqc(D_{k,x}^{(i,1)})$ and $\nqc(A)$ respectively.
Since $d_{ii'} \geq -\frac{3}{2}$ and $k \geq 6$, we have $x +d_{ii'}  \neq x - kd_i + s$ for any half integer $s$ between $-\frac{3}{2}$ and $\frac{1}{2}$. By Corollary \ref{cor: l-weight Demazure}, neither $A_{i',x+d_{ii'}}^{-1}$ nor $A_{i,x}^{-1} A_{i',x+d_{ii'}}^{-1}$ is a monomial in $\nqc(D_{k,x}^{(i,1)})$, so at least one of them is a monomial in $\nqc(A)$. Then $A_{i',x+d_{ii'}}^{-1}$ has to be a monomial in $\nqc(A)$. One obtains $j$ such that $c_{ij} = -3$ and $d_{ii'} = -\frac{1}{2} $. This implies that $\Glie$ is of type $G_2$ and $d_i = 1$, but then $c_{ii'} = 2 d_{ii'} = -1$ is an entry of the Cartan matrix of $\Glie$, absurd.
\end{proof}

\begin{theorem}  \label{thm: main}
Let $(i, k, x, y) \in I \times \BC^3$. If $k \notin \frac{1}{2} \BZ$, then in $K_0(\BGG)$ holds
\begin{multline} \label{equ: TQ}
\quad \quad \quad [\CM_{k,x}^{(i)}] [\CL(\frac{\Psi_{i,x}}{\Psi_{i,y}})] =  [\SL(\frac{\Psi_{i,x+d_i}}{\Psi_{i,y}})] \prod_{j: c_{ij} < 0} [\SL(\frac{\Psi_{j,x+d_{ij}}}{\Psi_{j,x+d_{ij}-kd_i}})]  \\
 +  [\SL(\frac{\Psi_{i,x-d_i}}{\Psi_{i,y}})] \prod_{j: c_{ij} < 0} [\SL(\frac{\Psi_{j,x-d_{ij}}}{\Psi_{j,x+d_{ij}-kd_i}})].
\end{multline}
\end{theorem}
When $\Glie = \mathfrak{sl}_2$, we recover Example \ref{example: KR sl2} as $\CM_{k,x}^{(1)} \cong \BC_x^2$ and $\SL(\frac{\Psi_{1,x}}{\Psi_{1,y}}) \cong \SL_y^x$.
\begin{proof}
Equation \eqref{equ: TQ} is equivalent to the normalized q-character formula:
\begin{equation} \label{equ: character formula}
\nqc(\CM_{k,x}^{(i)}) = (1 + A_{i,x}^{-1}) \prod_{j: c_{ij} < 0} \nqc(\SL(\frac{\Psi_{j,x+d_{ij}}}{\Psi_{j,x+d_{ij}-kd_i}})).
\end{equation}
This formula does not involve $y \in \BC$. So we may assume without loss of generality that neither $y - x$ nor $y - x + kd_i$ belongs to $\frac{1}{2} \BZ$. 

We follow closely the proof of \cite[Theorem 5.1]{Z4}. Introduce tensor products of asymptotic/KR modules (arbitrary order of tensor products):
\begin{gather*}
S^{\pm} := \SL(\frac{\Psi_{i,x\pm d_i}}{\Psi_{i,y}}) \otimes  \bigotimes_{j: c_{ij} < 0} \SL(\frac{\Psi_{j,x\pm d_{ij}}}{\Psi_{j,x+d_{ij}-kd_i}}), \\
T_n := \bigotimes_{j: c_{ij}<0} W_{n,x+d_{ij}-kd_i}^{(j)}, \quad S_n := \ W_{n,y}^{(i)} \otimes T_n  \quad \mathrm{for}\ n \in \BZ_{>0}.
\end{gather*}
One views $S^+, S^-$ and $S_n$ as analogs of $S^0, S^1$ and $S_n'$ in the proof of Claim 1 in \cite[Theorem 5.1]{Z4}. Let $\Bs^{\pm}, \Bt_n$ and $\Bs_n$ be the obvious highest $\ell$-weights of $S^{\pm}, T_n$ and $S_n$ as products of highest $\ell$-weights of asymptotic/KR modules. Then $\Bs^- =  \Bs^+A_{i,x}^{-1}$.

{\bf Step 1: irreducibility of the tensor products.}

Any $\ell$-weight of $S_n$ is of the form $\Be' \Be''$ where $\Be' $ and $\Be''$ are $\ell$-weights of $W_{n,y}^{(i)}$ and $T_n$ respectively. By Theorem \ref{thm: l-weights KR}, $\Be'$ is a monomial in the $Y_{j,b}^{\pm 1}$ with $b \in y + \frac{1}{2} \BZ$, while $\Be''$ is such a monomial but with $b \in x-kd_i + \frac{1}{2} \BZ$. Since $x-kd_i - y \notin \frac{1}{2} \BZ$, the $\ell$-weight $\Be' \Be''$ is dominant if and only if both $\Be'$ and $\Be''$ are dominant. 

Suppose $\Be'' \neq \Bt_n$. We prove that it is right-negative. 
By Theorem \ref{thm: l-weights KR} there exists $i' \in I$ such that $c_{ii'} < 0$ and $\Be'' \in \Bt_n A_{i', x+d_{ii'} - kd_i }^{-1} \lCQ^-_{x-kd_i}$. It suffices to show that $\Bt_n A_{i',x+d_{ii'} - kd_i}^{-1}$ is right-negative. By Equation \eqref{equ: KR dominant}, $\Bt_n$ is of the form
$$ \prod_{j: c_{ij} < 0} Y_{j,x + d_{ij} + \frac{1}{2}d_j-kd_i}^{\uparrow}. $$
Here $Y_{j,a}^{\uparrow}$ for $(j,a) \in I \times \BC$ means a certain product of $Y_{j,a}$ with the $Y_{j,b}$ such that $b  \in a + \frac{1}{2} \BZ_{>0}$. By Example \ref{example: simple roots right-negative}, the right-negativity of $A_{i',a}^{-1}$ for $a \in \BC$ arises from the negative power $Y_{i',a-\frac{1}{2}d_{i'}}^{-1}$. So we are led to prove that for all $j \in I$ with $c_{ij} < 0$ the following half integer is strictly positive:
$$ (x+d_{ij} + \frac{1}{2}d_j-kd_i) - (x+d_{ii'} - kd_i- \frac{1}{2}d_{i'}) =  d_{ij} + \frac{1}{2}d_j - d_{ii'} + \frac{1}{2}d_{i'} > 0. $$
Assume the contrary: the left-hand side as a half integer is non-positive. Since $d_j$ and $d_{i'}$ are strictly positive integers, we must have $i' \neq j$.  Next,
$$ \frac{1}{2}d_j (c_{ji} + 1)  = d_{ij}+ \frac{1}{2}d_j \leq  d_{ii'} - \frac{1}{2}d_{i'} = \frac{1}{2} d_{i'}(c_{i'i} - 1) \leq - d_{i'} \leq -1. $$
So $c_{ji} < - 1$. This implies $d_j = 1$ and $c_{ji} \leq -3$. The Lie algebra $\Glie$ must be of rank two and of type $G_2$, but $i, i', j$ are two-by-two distinct, absurd. 

Similarly, if $\Be' \neq \Bw_{n,y}^{(i)}$, then $\Be'$ is right-negative. $S_n$ admits only one dominant $\ell$-weight and must be irreducible. View it as an irreducible sub-quotient of  
\begin{align*}
L(\Bs^{\pm}) \otimes \SL(\frac{\Psi_{i,y+nd_i}}{\Psi_{i,x\pm d_i}}) \otimes  \bigotimes_{j: c_{ij} < 0} \SL(\frac{\Psi_{j,x+d_{ij}-kd_i+nd_j}}{\Psi_{j,x\pm d_{ij}}}).
\end{align*}
By assumption, $\lCQ_x^-$ and $\lCQ_y^- \lCQ_{x-kd_i}^-$ have trivial intersection. We obtain as in the proof of Claim 1 in \cite[Theorem 5.1]{Z4} that all monomials (counted with multiplicity) in  the normalized q-character $S_n$ must appear in that of $L(\Bs^{\pm})$, i.e., $\nqc(S_n)$ is bounded above by $\nqc(L(\Bs^{\pm}))$. As $n$ goes to infinity, $\nqc(S_n)$ converges to $\nqc(S^{\pm})$, and so the latter is bounded above by $\nqc(L(\Bs^{\pm}))$. Comparing highest $\ell$-weights, we conclude that $L(\Bs^{\pm}) \cong S^{\pm}$, namely, $S^{\pm}$ are also irreducible.

Claim 2 of \cite[Theorem 5.1]{Z4} adapted directly to the present situation, $S^{\pm}$ appear as irreducible sub-quotients of the tensor product $\CM_{k,x}^{(i)} \otimes \SL(\frac{\Psi_{i,x}}{\Psi_{i,y}})$, and $\nqc(\CM_{k,x}^{(i)})$ is bounded below by the right-hand side of Equation \eqref{equ: character formula}. 

{\bf Step 2: upper bound for the normalized q-character.}

 For $n \in 6 \BZ_{>0}$, view $\CM_{k,x}^{(i)}$ as an irreducible sub-quotient of $\CM_{n,x}^{(i)} \otimes B_n$, where $B_n$ is the following tensor product of asymptotic modules:
\begin{align*}
B_n &=   \bigotimes_{j: c_{ij} < 0}\SL(\frac{\Psi_{j,x+d_{ij}-nd_i}}{\Psi_{j,x+d_{ij}-kd_i}}).
\end{align*}
Fix an $\ell$-weight $\Bm_{k,x}^{(i)} \Be$ of $\CM_{k,x}^{(i)}$. It is of the form $\Be = \Be_n^M \Be_n^B$ such that $\Be_n^M \in \lCQ_x^-$ and $\Be_n^B \in \lCQ_{x-kd_i}^-$ are monomials in $\nqc(\CM_{n,x}^{(i)})$ and $\nqc(B_n)$ respectively; such a factorization is unique because $\lCQ_x^-$ and $\lCQ_{x-kd_i}^-$ intersect trivially. 
Let $n$ be so large that none of the $A_{j,x+d_{ij}-nd_i}^{-1}$ for $j \in I$ appears as a factor of the monomial $\Be$. By Corollary \ref{cor: l-weight prime Demazure}, either $\Be_n^M = 1$ or $\Be_n^M = A_{i,x}^{-1}$. In both cases, the $\ell$-weight space of $\CM_{n,x}^{(i)}$ of $\ell$-weight $\Bm_{n,x}^{(i)} \Be_n^M$ is one-dimensional, as seen from the proof of Corollary \ref{cor: l-weight prime Demazure}. This implies that $\nqc(\CM_{k,x}^{(i)})$ is bounded above by $\nqc(B_n) (1+A_{i,x}^{-1})$, which is the right-hand side of Equation \eqref{equ: character formula}. Together with the lower bound at Step 1, we obtain Equation \eqref{equ: character formula}. The proof of Theorem \ref{thm: main} is completed. 
\end{proof}

Let us be at Step 1 of the proof. $S^{+}$ being irreducible, its first tensor factor must be irreducible. We obtain the following result.

\begin{cor}
For $i \in I$ and $x, y \in \BC$ such that $y - x \notin \frac{1}{2}\BZ$, the asymptotic module $\SL(\frac{\Psi_{i,y}}{\Psi_{i,x}})$ is irreducible.
\end{cor}

\begin{rem}  \label{rem: factorization}
For $(i,k,x) \in I \times \BC^2$ define
 \begin{align*}
\Bn_{k,x}^{(i)} &:= \prod_{j: c_{ij} = -2} \frac{\Psi_{j,x}}{\Psi_{j,x-k} }  \times \prod_{j: c_{ij} = -3} \frac{\Psi_{j,x+\frac{1}{2}} \Psi_{j,x-\frac{1}{2}}}{\Psi_{j,x+\frac{1}{2}-k} \Psi_{j,x-\frac{1}{2}-k}}.
\end{align*}
If $k \in 6 \BZ_{>0}$, then $\Bm_{k,x}^{(i)} \Bn_{k,x}^{(i)} = \Bd_{k,x}^{(i,1)}$ by Equation \eqref{equ: Demazure weight}. Using Corollary \ref{cor: l-weight Demazure} one can modify the proof of Theorem \ref{thm: main} to obtain a tensor product factorization: 
$$ L(\Bm_{k,x}^{(i)} \Bn_{k,x}^{(i)} ) \cong L(\Bm_{k,x}^{(i)}) \otimes L(\Bn_{k,x}^{(i)}) \quad \mathrm{if\ either}\ k \notin \frac{1}{2} \BZ \ \mathrm{or}\ k \in 6 \BZ_{>0}. $$
%Similar factorization should hold for the module $X_{i,a}$ in \cite[Definition 3.1]{FH2} (by singling out $\Psi_{j,a}$ for $c_{ij} = -2$ and $\Psi_{j,aq}\Psi_{j,aq^{-1}}$ for $c_{ij} = -3$ in {\it loc.cit.}).
\end{rem}

\begin{example} 
Let $\Glie = \mathfrak{sl}_3$ and $i = 1$. Suppose $k \in \BC \setminus \frac{1}{2} \BZ$. The KR module $W_{1,0}^{(1)}$ is a vector representation of $Y_{\hbar}(\mathfrak{sl}_3)$ on $\BC^3$. From Theorem \ref{thm: l-weights KR} we get:
\begin{gather*}
\qc(W_{1,0}^{(1)}) = \frac{\Psi_{1,1}}{\Psi_{1,0}}(1 + A_{1,0}^{-1} + A_{1,0}^{-1} A_{2,-\frac{1}{2}}^{-1})  = \frac{\Psi_{1,1}}{\Psi_{1,0}} + \frac{\Psi_{1,-1}}{\Psi_{1,0}} \frac{\Psi_{2,\frac{1}{2}}}{\Psi_{2,-\frac{1}{2}}} + \frac{\Psi_{2,-\frac{3}{2}}}{\Psi_{2,-\frac{1}{2}}}, \\
[W_{1,0}^{(1)}] = \frac{[\SL(\frac{\Psi_{1,1}}{\Psi_{1,0}})]}{[\SL(\frac{\Psi_{1,0}}{\Psi_{1,0}})]} + \frac{[\SL(\frac{\Psi_{1,-1}}{\Psi_{1,0}})]}{[\SL(\frac{\Psi_{1,0}}{\Psi_{1,0}})]}\frac{[\SL(\frac{\Psi_{2,\frac{1}{2}}}{\Psi_{2,0}})]}{[\SL(\frac{\Psi_{2,-\frac{1}{2}}}{\Psi_{2,0}})]}  + \frac{[\SL(\frac{\Psi_{2,-\frac{3}{2}}}{\Psi_{2,0}})]}{[\SL(\frac{\Psi_{2,-\frac{1}{2}}}{\Psi_{2,0}})]}, \\
[\CM_{k,x}^{(1)}] [\SL(\frac{\Psi_{1,x}}{\Psi_{1,y}})] = [\SL(\frac{\Psi_{1,x+1}}{\Psi_{1,y}})]  [\SL(\frac{\Psi_{2,x-\frac{1}{2}}}{\Psi_{2, x-\frac{1}{2}-k}})] +  [\SL(\frac{\Psi_{1,x-1}}{\Psi_{1,y}})]  [\SL(\frac{\Psi_{2,x+\frac{1}{2}}}{\Psi_{2, x-\frac{1}{2}-k}})].
\end{gather*}
One can show that as $\mathfrak{sl}_3$-module $\CM_{k,x}^{(1)}$ is irreducible of highest weight $\varpi_1 + k \varpi_2$.
\end{example}

\section{Prefundamental modules and shifted Yangians}   \label{sec: prefund}

In this section we construct modules of highest $\ell$-weight $\Psi_{i,x}^{\pm1}$ for $(i,x) \in I \times \BC$. These will be modules over the so-called {\it shifted Yangians}. We assume $\hbar = 1$. 

Let $\varpi_i^{\vee} \in \Hlie$ for $i \in I$ be the $i$th fundamental coweight; it is determined by the equations $\langle \varpi_i^{\vee}, \alpha_j\rangle = \delta_{ij}$ for $j \in I$, where $\langle, \rangle$ denotes the evaluation map $\Hlie \times \Hlie^* \longrightarrow \BC$. We get the coweight lattice $\BQ^{\vee} = \oplus_{i\in I} \BZ \varpi_i^{\vee}$. 

For a coweight $\mu \in \BQ^{\vee}$, recall the {\it shifted Yangian} $Y_{\mu}(\Glie)$ from \cite[Definition B.2]{BFN}: it is an algebra  generated by $E_i^{(m)}, F_i^{(m)}, H_i^{(p)}$ for $(i,m,p) \in I \times \BZ_{>0} \times \BZ$ subject to the following relations for $(i,j, m, n, p, q) \in I^2 \times \BZ_{>0}^2 \times \BZ^2$:
\begin{gather*}
[H_i^{(p)}, H_j^{(q)}] = 0,\quad [E_i^{(m)}, F_j^{(n)}] = \delta_{ij} H_i^{(m+n-1)}, \\
[H_i^{(p+1)}, E_j^{(m)}] - [H_i^{(p)}, E_j^{(m+1)}] = d_{ij}  (H_i^{(p)} E_j^{(m)} + E_j^{(m)} H_i^{(p)}), \\
[H_i^{(p+1)}, F_j^{(m)}] - [H_i^{(p)}, F_j^{(m+1)}] = -d_{ij} (H_i^{(p)} F_j^{(m)} + F_j^{(m)} H_i^{(p)}),  \\
[E_i^{(m+1)}, E_j^{(n)}] - [E_i^{(m)}, E_j^{(n+1)}] = d_{ij} (E_i^{(m)} E_j^{(n)} + E_j^{(n)} E_i^{(m)}),  \\
[F_i^{(m+1)}, F_j^{(n)}] - [F_i^{(m)}, F_j^{(n+1)}] = -d_{ij} (F_i^{(m)} F_j^{(n)} + F_j^{(n)} F_i^{(m)}), \\
\mathrm{ad}_{E_i^{(1)}}^{1-c_{ij}} (E_j^{(m)}) = 0 = \mathrm{ad}_{F_i^{(1)}}^{1-c_{ij}} (F_j^{(m)}) \quad \mathrm{if}\ i \neq j, \\
H_i^{(-\langle\mu, \alpha_i\rangle) } = 1 \quad \mathrm{and}\quad H_i^{(p)} = 0 \quad \mathrm{for}\  p < -\langle \mu, \alpha_i\rangle.
\end{gather*}
When the relations at the last line are dropped, one gets the {\it Cartan doubled Yangian} $Y_{\infty}(\Glie)$; see \cite[Definition B.1]{BFN}. Let $\pi_{\mu}: Y_{\infty}(\Glie) \longrightarrow Y_{\mu}(\Glie)$ be the quotient map. Define the generating series with coefficients in $Y_{\infty}(\Glie)$ or in $Y_{\mu}(\Glie)$ for $i \in I$: 
$$ E_i(u) := \sum_{m=1}^{+\infty} E_i^{(m)} u^{-m},\quad F_i(u) := \sum_{m=1}^{+\infty} F_i^{(m)} u^{-m},\quad H_i(u) := \sum_{p=-\infty}^{+\infty} H_i^{(p)} u^{-p}. $$
One has the spectral parameter shift, an algebra automorphism $\tau_x$ on $Y_{\mu}(\Glie)$:
$$ \tau_x: \quad E_i(u) \mapsto E_i(u+x),\quad F_i(u) \mapsto F_i(u+x),\quad H_i(u) \mapsto H_i(u+x). $$

Fix $i \in I$. We relate $Y_{\infty}(\Glie)$ to the ordinary Yangian $Y_{\hbar = 1}(\Glie)$. For $l \in \BC^{\times}$ the following defines an algebra morphism $\theta_{i,l}: Y_{\infty}(\Glie) \longrightarrow Y_{\hbar=1}(\Glie)$,
\begin{gather*}
E_j(u) \mapsto x_j^+(u),\quad F_j(u) \mapsto (\frac{1}{ld_i})^{\delta_{ij}} x_j^-(u),\quad H_j(u) \mapsto (\frac{1}{ld_i})^{\delta_{ij}} \xi_j(u).
\end{gather*}
In particular, $\theta_{i,d_i^{-1}}$ induces an isomorphism $Y_{\mu = 0}(\Glie) \cong Y_{\hbar = 1}(\Glie)$.

Let us be in the situation of Section \ref{sec: asym} with $\hbar = 1$. We have an inductive system of vector spaces $(W_{l,0}^{(i)}, F_{k,l})$. Each of the $W_{l,0}^{(i)}$ becomes an $Y_{\infty}(\Glie)$-module after pullback by $\theta_{i,l}$. Let us rewrite Lemma \ref{lem: compt with x+} and Equations \eqref{rel: compt with xi j}--\eqref{rel: compt with x i} in terms of generators of $Y_{\infty}(\Glie)$. For $j \in I$ and $l < l+1 < k$ we have
\begin{gather*} 
E_j(u) F_{k,l} = F_{k,l} E_j(u);  \\
H_j(u) F_{k,l} = F_{k,l} H_j(u), \quad  F_j(u)F_{k,l} = F_{k,l} F_j(u) \quad \mathrm{if}\ j \neq i; \\
H_i(u) F_{k,l} = \left(\frac{ld_i}{u+ld_i} + \frac{uld_i}{u+ld_i} (kd_i)^{-1}  \right) F_{k,l} H_i(u); \\
F_i(u) F_{k,l} = F_{k,l+1}  \left(B_i^l(u) + A_i^l(u) (kd_i)^{-1}  \right).
\end{gather*}
For each generator $t$ of $Y_{\infty}(\Glie)$ among the $\{E_j^{(m)}, F_j^{(m)}, H_j^{(p)} \}$ and for $l \in \BZ_{>0}$, there exist uniquely two linear maps $C_t^l$ and $D_t^l$ from $W_{l,0}^{(i)}$ to $W_{l+1,0}^{(i)}$ such that:
$$ t F_{k,l} = F_{k,l+1} (C_t^l + D_t^l (kd_i)^{-1} ) \quad \mathrm{for}\ k > l+1. $$
As in Section \ref{sec: asym}, the $C_t^l$ and $D_t^l$ are morphisms of inductive systems of vector spaces. Let $C_t$ and $D_t$ be their inductive limits, which are linear endomorphisms on $W_{\infty}^{(i)}$. Then the assignment $t \mapsto C_t$ defines a representation $\rho_i$ of $Y_{\infty}(\Glie)$ on $W_{\infty}^{(i)}$.

\begin{prop}  \label{prop: negative pre}
The representation $\rho_i$ of $Y_{\infty}(\Glie)$ factorizes through the projection $\pi_{-\varpi_i^{\vee}}: Y_{\infty}(\Glie) \longrightarrow Y_{-\varpi_i^{\vee}}(\Glie)$ and we get an $Y_{-\varpi_i^{\vee}}(\Glie)$-module structure on $W_{\infty}^{(i)}$. For $x \in \BC$, the negative prefundamental module, denoted by $L_{i,x}^-$, is defined to be the pullback of the $Y_{-\varpi_i^{\vee}}(\Glie)$-module $W_{\infty}^{(i)}$ by $\tau_x$.
\end{prop}
\begin{proof}
It suffices to show that $\rho_i(H_j(u))$ for $j \in I$ is a power series of $u^{-1}$ with leading term $u^{-\delta_{ij}}$.  Set $C_{H_j(u)}^l:=\sum_p C_{H_j^{(p)}}^l u^{-p}$. If $j \neq i$, then:
\begin{align*}
C_{H_j(u)}^l &= F_{l+1,l} H_j(u) = F_{l+1,l} \xi_j(u) \in F_{l+1,l} + u^{-1} \mathrm{Hom}(W_{l,0}^{(i)}, W_{l+1,0}^{(i)}) [[u^{-1}]], \\
C_{H_i(u)}^l &= F_{l+1,l} \frac{ld_i}{u+ld_i} H_i(u) = F_{l+1,l} \frac{1}{u+ld_i} \xi_i(u) \\
&= F_{l+1,l} (\sum_{n=0}^{+\infty} u^{-n-1} (-ld_i)^{n}) (1 + \sum_{n=0}^{+\infty} \xi_{i,n} u^{-n-1})  \\
& \in u^{-1} F_{l+1,l} + u^{-2}\mathrm{Hom}(W_{l,0}^{(i)}, W_{l+1,0}^{(i)}) [[u^{-1}]].
\end{align*}
The inductive limit of the $F_{l+1,l}$ is the identity map on $W_{\infty}^{(i)}$. 
\end{proof}

The construction of $L_{i,x}^-$ is similar to that in \cite[\S 4.2]{HJ}: ordinary Yangian, Cartan doubled Yangian and shifted Yangian correspond to quantum affine algebra, asymptotic algebra and Borel algebra in \cite[\S 2]{HJ}. 

The shifted Yangian $Y_{\mu}(\Glie)$ is weight graded based on the adjoint action of the $H_i^{(-\langle \mu, \alpha_i\rangle + 1)}$ so that $E_i^{(m)}, F_i^{(m)}$ and $H_i^{(p)}$ are of weight $\alpha_i, -\alpha_i$ and $0$: 
$$ [H_i^{(-\langle \mu, \alpha_i\rangle + 1)}, E_j^{(m)}] = (\alpha_i, \alpha_j) E_j^{(m)},  \quad [H_i^{(-\langle \mu, \alpha_i\rangle + 1)}, F_j^{(m)}] = -(\alpha_i, \alpha_j) F_j^{(m)}. $$
One defines weight, (highest) $\ell$-weight and q-character for an $Y_{\mu}(\Glie)$-module and category $\BGG_{\mu}$ of $Y_{\mu}(\Glie)$-modules as in Section \ref{sec: pre} by replacing $\xi_{i,0}, x_i^+(u)$ and $\xi_i(u)$ with $H_i^{(-\langle \mu, \alpha_i\rangle + 1)}, E_i^+(u)$ and $H_i(u)$. 

As an example, $L_{i,x}^-$ is in category $\BGG_{-\varpi_i^{\vee}}$, it contains a highest $\ell$-weight vector of $\ell$-weight $\Psi_{i,x}^{-1}$, and its q-character is 
$$ \qc(L_{i,x}^-) = \frac{1}{\Psi_{i,x}} \times \lim_{l \rightarrow \infty} \nqc(W_{l,x}^{(i)}). $$

\begin{rem}
The shifted Yangian $Y_{\varpi_i^{\vee}}(\Glie)$ has a one-dimensional representation:
$$ E_j(u) = 0 = F_j(u), \quad H_j(u) = (u+x)^{\delta_{ij}}. $$
This module is of highest $\ell$-weight $\Psi_{i,x}$, and so is the positive prefundamental module $L_{i,x}^+$. Note however that in the case of quantum affine algebras positive and negative prefundamental modules have the same character \cite[Theorem 6.3]{HJ}.
\end{rem}

\begin{rem}
In \cite[\S 4]{Fin}, there is a remarkable family of algebra morphisms (for $\Glie$ simply-laced, but this assumption may be removed in view of \cite[Remark 3.2]{Fin}) 
$$\Delta_{\mu_1,\mu_2}: Y_{\mu_1+\mu_2}(\Glie) \longrightarrow Y_{\mu_1}(\Glie) \otimes Y_{\mu_2}(\Glie)\quad \mathrm{for}\ \mu_1, \mu_2 \in \BQ^{\vee} $$
which satisfies co-associativity for anti-dominant coweights. This implies that the direct sum of the categories $\BGG_{\mu}$ for $\mu$ anti-dominant forms a monoidal category. Such a category contains all the $Y_{\hbar=1}(\Glie)$-modules in category $\BGG$ and the negative prefundamental modules, so we expect it to have similar properties as the category $\BGG^-$ in \cite[Definition 3.9]{HL}. Also irreducibility of $L_{i,x}^-$ can be proved along the line of \cite[Theorem 6.1]{HJ} using the monoidal structure. 
\end{rem}
Note that a similar monoidal category for the degenerate Yangian of $\mathfrak{gl}_2$ appeared in \cite[Appendix A]{FZ}.
As indicated in the introduction,  there are  prefundamental modules over degenerate Yangians for arbitrary $\Glie$ based on the R-matrix realization \cite{CWY, Guay, JLM, W}. Their relationship with our modules is to be clarified.

\appendix
\section{Three-term relations for quantum affine algebras} 
Set $\iota = \sqrt{-1} \in \BC$. Assume $\hbar \notin \mathbb{Q}$, so that $q := e^{\pi \iota \hbar} \in \BC^{\times}$ is not a root of unity. For $x \in \BC$ and $i,j \in I$, set $q^x := e^{\pi \iota \hbar x},\ q_{ij} := q^{d_{ij}}$ and $q_i := q^{d_i}$. We produce three-term relations in category $\hat{\BGG}$ of representations of the quantum affine algebra $U_q(\widehat{\Glie})$ in \cite[Definition 3.3]{MY}. Such a category appeared first in \cite{He}.

Let $x_{i,r}^{\pm}, \phi_{i,\pm n}^{\pm}$, for $(i, r, n) \in I \times \BZ \times \BZ_{\geq 0}$, be the Drinfeld generators of $U_q(\widehat{\Glie})$.  The $\phi_{i,\pm n}^{\pm}$ mutually commute and $\phi_{i,0}^+ \phi_{i,0}^- = 1$, whose spectral decomposition on a module defines the notion of $\ell$-weight.

As in \cite[Definition 3.5]{MY}, let $\mathcal{R}$ be the set of $I$-tuples $(\Be_i(z))_{i \in I}$ of rational functions of $z$ such that each $\Be_i(z)$ is regular at $0, \infty$ and $\Be_i(0) \Be_i(\infty) = 1$. Equivalently, each $\Be_i(z)$ is a finite product of rational functions of the form $\frac{c-zc^{-1}}{a-za^{-1}}$ for $c, a \in \BC^{\times}$. We take Taylor expansions around $0, \infty$ to obtain formal power series in $z^{\pm 1}$:
$$ \sum_{n\geq 0} \Be_{i,n}^+ z^n = \Be_i(u) = \sum_{n \geq 0} \Be_{i,-n}^- z^{-n}.  $$
View $\BC[[z]]^I$ as a monoid by component-wise multiplication. For $(i, a) \in I \times \BC^{\times}$ we define an invertible element $\Phi_{i,a}$ of this monoid by
$$ (\Phi_{i,a})_j(z) = 1 \quad \mathrm{if}\ j \neq i,\quad (\Phi_{i,a})_i(z) = a - z a^{-1}.   $$
This is a scalar multiple of $\Psi_{i,a^{-2}}$ in \cite[Definition 3.7]{HJ}, and equals $X_{i,a^{-1}}$ in \cite[\S 4.1]{Jimbo}. Now $\mathcal{R}$ is characterized as the subgroup of $\BC[[u]]^I$ generated by the ratios $\frac{\Phi_{i,c}}{\Phi_{i,a}}$ for $(i, a, c) \in I \times \BC^{\times} \times \BC^{\times}$. 

By \cite[Theorem 3.6]{MY}, $\mathcal{R}$ is in bijection with the isomorphism classes of irreducible modules in category $\widehat{\BGG}$. To $\Bd \in \mathcal{R}$ is attached an irreducible $U_q(\widehat{\Glie})$-module, denoted by $L(\Bd)$, which is generated by a vector $\omega$ subject to relations:
$$ x_{i,r}^+ \omega = 0,\quad \phi_{i,\pm n}^{\pm} \omega = \Bd_{i,\pm n}^{\pm} \omega \quad \mathrm{for}\ (i, r, n) \in I \times \BZ \times \BZ_{\geq 0}. $$
Fix $(i, a, c) \in I \times \BC^{\times} \times \BC^{\times}$. There is a series of finite-dimensional irreducible modules $L(\frac{\Phi_{i,aq_i^k}}{\Phi_{i,a}})$ for $k \in \BZ_{\geq 0}$, called Kirillov--Reshetikhin modules. It is explained in \cite[Appendix A]{Z2} that this series admits analytic continuation: as $k$ tends to infinity by replacing $q_i^k$ with  $c$ one obtains an asymptotic module denoted by $\SL(\frac{\Phi_{i,ac}}{\Phi_{i,a}})$.

\begin{theorem} \label{them: three-term quantum}
Let $(i, a, c, k) \in I \times \BC^{\times} \times \BC^{\times} \times \BC$  such that $q^k \notin q^{\frac{1}{2} \BZ}$. In the Grothendieck ring of category $\hat{\BGG}$ we have 
\begin{multline} \label{equ: TQ quantum}
\quad \quad \quad [L(\frac{\Phi_{i,aq_i}}{\Phi_{i,a}} \prod_{j: c_{ij} < 0} \frac{\Phi_{j,aq_{ij}}}{\Phi_{j,aq_{ij} q_i^{-k}}})] [\CL(\frac{\Phi_{i,a}}{\Phi_{i,c}})] =  \\ 
[\SL(\frac{\Phi_{i,aq_i}}{\Phi_{i,c}})] \prod_{j: c_{ij} < 0} [\SL(\frac{\Phi_{j,aq_{ij}}}{\Phi_{j,aq_{ij}q_i^{-k}}})]  
 +  [\SL(\frac{\Phi_{i,aq_i^{-1}}}{\Phi_{i,c}})] \prod_{j: c_{ij} < 0} [\SL(\frac{\Phi_{j,aq_{ij}^{-1}}}{\Phi_{j,aq_{ij}q_i^{-k}}})].
\end{multline}
\end{theorem}

The proof of the theorem is almost identical to that of Theorem \ref{thm: main}.

Theorems \ref{thm: main} and \ref{them: three-term quantum} should be included in the framework of \cite{GTL} relating Yangians to quantum affine algebras. Note that \cite{GTL} only involved integrable modules, for which Serre relations hold for free, while our asymptotic modules are non integrable. We expect that the functor $\Gamma_{\Pi}$ of \cite{GTL} can be extended to a functor $\Gamma: \BGG \longrightarrow \widehat{\BGG}$ from $Y_{\hbar}(\Glie)$-modules to $U_q(\widehat{\Glie})$-modules so that 
$$ \Gamma(L(\prod_{(i,a,b) \in X} \frac{\Psi_{i,a}}{\Psi_{i,b}} )) \cong  L(\prod_{(i,a,b) \in X} \frac{\Phi_{i,q^a}}{\Phi_{i,q^b}} ) $$
for any finite subset $X$ of $I \times \BC^2$ and the q-characters are preserved after the replacement $\Psi_{i,a} \longleftrightarrow \Phi_{i,q^a}$. This agrees with the finite-dimensional case in view of Equations \eqref{def: fund l-weights}, \eqref{equ: GTL} and the identification
$\mathcal{Y}_{i,q^{-2a}} = \Phi_{i,q^{a+\frac{1}{2}d_i}} \Phi_{i,q^{a-\frac{1}{2}d_i}}^{-1}$.

One may view $\hat{\BGG}$ as a subcategory of the category of modules over the upper Borel subalgebra introduced by Hernandez--Jimbo \cite[Definition 3.8]{HJ}. This larger category admits irreducible modules of highest $\ell$-weights $\Phi_{i,a}$ for all $(i,a) \in I \times \BC^{\times}$. Equation \eqref{equ: TQ quantum} becomes \cite[Eq.(6.13)]{HL} and \cite[Eq.(1.3)]{Jimbo} when one removes all the $\Phi_{i,c}$ and $\Phi_{j,aq_{ij}q_i^{-k}}$ in the denominators and replaces $\SL$ by $L$.

Similar three-term relations have been established for prefundamental modules over quantum toroidal $\mathfrak{gl}_1$ \cite[Eq.(4.24)]{Jimbo2} and for asymptotic modules over quantum affine superalgebras of type A \cite[Eq.(5.30)]{Z3}.

\end{document}